\documentclass{article}

\usepackage[english]{babel}

\usepackage[letterpaper,top=2cm,bottom=2cm,left=3cm,right=3cm,marginparwidth=1.75cm]{geometry}

\usepackage{amsmath}
\usepackage{amsthm}
\usepackage{graphicx}
\usepackage[colorlinks=true, allcolors=blue]{hyperref}
\usepackage{amssymb}
\usepackage{tikz}
\usetikzlibrary{positioning}
\usepackage{cmap}
\usepackage[utf8]{inputenc}
\usepackage[T1]{fontenc}
\usepackage{tcolorbox}
\usepackage{enumitem}
\usepackage{multicol}
\usepackage{caption}
\renewcommand{\arraystretch}{1.5} 
\usepackage[table]{xcolor} 
\usepackage{array} 
\usetikzlibrary{arrows}
\usetikzlibrary{arrows.meta}
\usepackage{subcaption}

\usepackage{svg}

\newtheorem{theorem}{Theorem}[section]
\newtheorem{corollary}[theorem]{Corollary}

\newtheorem{lemma}[theorem]{Lemma}

\usepackage{xcolor}

\newcommand{\R}[0]{\mathbb{R}}
\newcommand{\Pb}[0]{\mathbb{P}}

\newcommand{\Sb}[0]{\mathbb{S}}

\newcommand{\diag}[0]{\mathtt{diag}}

\newcommand{\gr}{\mathtt{gr}}
\newcommand{\id}{\mathtt{id}}

\title{Geometric Realizations with Strong Self-Duality \\[0.5em] \large Part I: Non-Bipartite Quadrangulations of the Projective Plane and Strongly Involutive Self-Dual Graphs}

\author{Gábor Damásdi\footnote{Supported by the ERC Advanced Grant no. 882971 GeoScape, by the János Bolyai Research Scholarship of the Hungarian Academy of Sciences, the STARTING-25 project no. 152199 and the EXCELLENCE-24 project no. 151504 Combinatorics and Geometry of the NRDI Fund.} \\ \small HUN-REN Alfr\'ed R\'enyi Institute of Mathematics,\\ \small ELTE Eötvös Loránd University, Budapest }

\begin{document}
\maketitle
\begin{abstract}
   In this series of articles, we study two graph classes and their various geometric representations. The first class consists of strongly involutive self-dual graphs (SISD graphs); the second consists of non-bipartite quadrangulations of the projective plane (NBQP graphs). These classes naturally arise in many geometric problems. For example, NBQP graphs are connected to diameter graphs in $\mathbb{R}^3$, tangency graphs of families of pairwise intersecting circles and pseudocircles, generalized thrackles, and graphs formed by the main diagonals of negatively self-polar polytopes. On the other hand, SISD graphs correspond to the skeletons of negatively self-polar polytopes, self-dual cones, and extremal ball-polytopes. 
    
    In Part I we show a natural one-to-one correspondence between non-bipartite quadrangulations of the projective plane and 2-connected strongly involutive self-dual maps. In particular, we show that the reduced vertex-face incidence structure of a 2-connected strongly involutive self-dual map is a non-bipartite quadrangulation of the projective plane, and every non-bipartite quadrangulation of the projective plane arises that way. This allows us to translate natural properties from one class to the other. These results lay the foundation for later parts of the series. 
\end{abstract}
\section{Introduction for the series}
One of the most joyful forms of discovery is when we find connections between seemingly different areas of life or mathematics. Bronowski writes about an interesting case of such discovery \cite{bronowski1956}. The native Sherpas, living on the slopes of Mount Everest, knew different faces of the mountain very well, but did not realize that they were looking at the same mountain. They even had different names for the same mountain. 

In a series of papers, we present a similar mountain in the world of discrete geometry. Without going into the technical details, here are some of the conjectures and theorems in which the peak of the mountain was spotted from a distance.

\begin{theorem}[Grünbaum's conjecture, solved by Ackerman et al. \cite{pseudodamasdi}]\label{thm:listgr}
    Every pairwise intersecting family of (pseudo) disks determines at most $2n-2$ touchings.
\end{theorem}

\begin{theorem}[Vázsonyi problem \cite{ERD46}, solved by Grünbaum, Heppes, and Straszewicz \cite{GRU56,HEP56, STR57}]\label{thm:listvazsony}
   Every finite point set in $\R^3$ determines at most $2n-2$ diameters.  
\end{theorem}

\begin{theorem}[Question of Lovász, Pach, and Szegedy \cite{lovasz1997conway}, solved by Cairns and Nikolayevsky \cite{CN00}]
Every generalized thrackle has at most $2n-2$ edges.    
\end{theorem}

\begin{theorem}[Conjecture of Kupitz \cite{KUP84}, solved by Valtr \cite{VALTR98}]
 Every geometric graph without avoiding edges has at most $2n-2$ edges.   
\end{theorem}

\begin{theorem}[Consequence of the Euler formula]
    A non-bipartite quadrangulation of the projective plane has $2n-2$ edges.
\end{theorem}

Our goal in this series is to show that the occurrence of $2n-2$ in so many results is not just a coincidence. We are seeing the same mountain -- the same graph class in this case -- from many different valleys. That is, not only is the maximal number of edges the same in these problems, but the same class of graphs also arises in them. Just as in the case of Everest, many connections were known or conjectured, but in some cases even researchers who worked on several of these problems missed the big picture. For example, Grünbaum solved Theorem \ref{thm:listvazsony} and conjectured Theorem \ref{thm:listgr}, but to the best of our knowledge never showed any connection between the problems. 

A common theme of the series is strong forms of self-duality. There are various notions of duality in discrete geometry, such as duality of planar graphs, duality of polytopes and duality of ball polytopes. For each of these we study structures that are self-dual in some strong sense. Interestingly, many of them show the same underlying combinatorial structures. We will see that strongly involutive self-dual maps correspond to skeletons of self-dual structures and non-bipartite quadrangulations of the projective plane encode the vertex-face incidence structure, see Figures \ref{fig:sisdintro} and \ref{fig:nbqpintro}, respectively. 

We have also created a webpage \cite{damasdi_selfdual_visualizations} for these results where we collected a large number of examples and visualizations. You can also find a number of  talks on the webpage that the author has given on the topic in the past few years. This is going to be a long journey. Stay with us and enjoy the many views of the mountain.

\begin{figure}[!ht]
    \centering
    
    \includegraphics[width=0.18\linewidth]{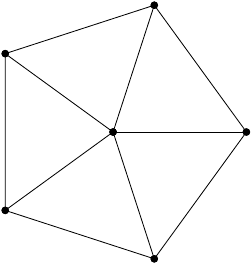}
    \includegraphics[width=0.18\linewidth]{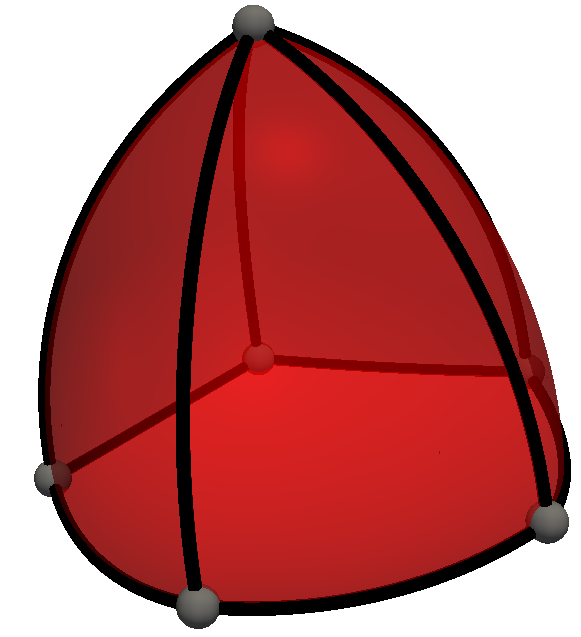}
    \includegraphics[width=0.18\linewidth]{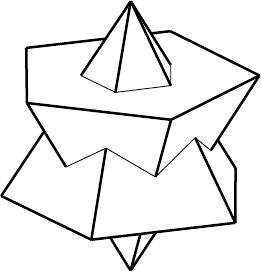}
    \caption{A strongly involutive self-dual graph, a Reuleaux polyhedron and a negatively self-polar polytope.}
    \label{fig:sisdintro}
\end{figure}

\begin{figure}[!ht]
    
    \centering
    \includegraphics[width=0.18\linewidth]{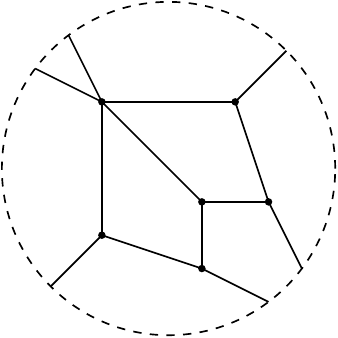}
    \includegraphics[width=0.18\linewidth]{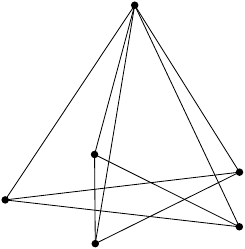}
    \includegraphics[width=0.18\linewidth]{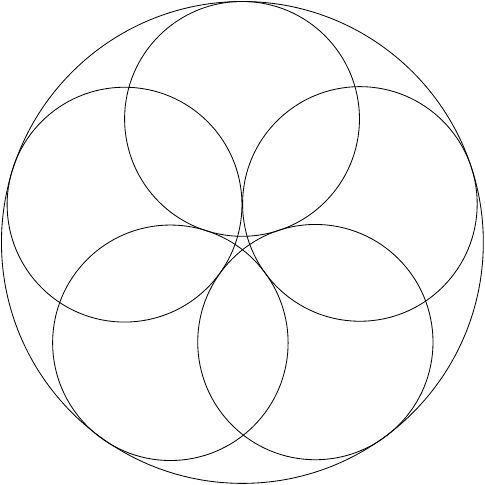}
    \includegraphics[width=0.18\linewidth]{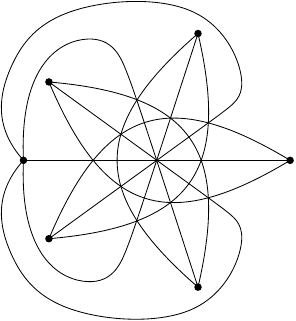}
   
    \caption{ A non-bipartite quadrangulation of the projective plane, a diameter graph in $\R^3$, a family of pairwise intersecting circles, a generalized thrackle graph.}
    \label{fig:nbqpintro}
\end{figure}

The planned parts of the series are as follows.

\begin{itemize}
    \item[Part I.] We show a natural one-to-one correspondence between non-bipartite quadrangulations of the projective plane and 2-connected strongly involutive self-dual maps. 
    \item[Part II.] We fully characterize diameter graphs in $\R^3$. Using this characterization and a novel argument in rigidity theory, we confirm a conjecture of Kupitz-Martini-Perles and a conjecture of Montejano-Pauli-Raggi-Roldán-Pensado on skeletons of Reuleaux polyhedrons. 
    \item[Part III.] We study tangency graphs of pairwise intersecting disk and pseudodisk families in the plane. We characterize the possible tangency graphs, extending the recent resolution of a conjecture of Grünbaum.
    \item[Part IV.]  We consider self-dual cones in $\R^4$ and negatively self-polar polytopes in $\R^3$. We confirm a conjecture of Gouveia-Lourenço and a conjecture of Kupitz-Martini-Perles by showing that the possible skeletons of negatively self-polar polytopes are the 3-connected strongly involutive self-dual maps.
\end{itemize}

\section{Introduction for Part I}

In this paper, we focus on building a connection between two important map classes, the strongly involutive self-dual maps and the non-bipartite quadrangulations of the projective plane. This will provide the combinatorial tools needed for later parts of the series.

  Let us start by explaining a well-known correspondence between planar maps, their duals and quadrangulations.  We consider planar maps as graphs embedded in the sphere, i.e., there is no distinguished outer face. Given a 2-connected planar map $G=(V,E,F)$, its dual map $G^*=(V^*,E^*,F^*)$ is constructed by placing a vertex in each face of $G$ and connecting neighboring faces through their common edges.

  The \emph{radial map $R(G)$} of a planar map is constructed as follows. The vertex set is the vertices and faces of $G$, and edges correspond to vertex-face incidences. We can naturally embed the radial map on the plane by placing each face-vertex in the corresponding face of $G$ and connecting it to incident vertices; see Figure \ref{fig:planarradial}. 

    \begin{figure}[!ht]
        \centering
        \includegraphics[width=0.2\linewidth]{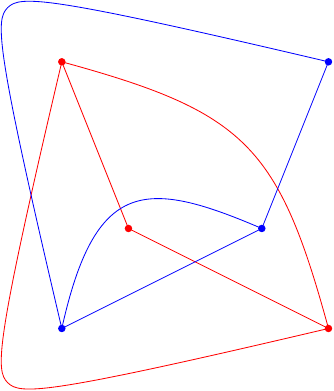}\hspace{20px}
        \includegraphics[width=0.2\linewidth]{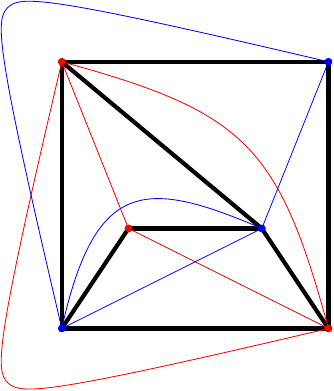}
        \caption{A planar map with its dual and the corresponding common radial graph.}
        \label{fig:planarradial}
    \end{figure}

  The radial map is always a quadrangulation. A planar map and its dual have the same radial map. Conversely, consider a planar quadrangulation. It is bipartite, hence we can construct two diagonal graphs by drawing in each quadrangle the diagonal whose endpoints belong to a given color class. The two diagonal maps are dual to each other and their radial map is the quadrangulation. This connection is very useful as it allows us to translate properties of planar maps to properties of quadrangulations and vice versa. See, for example, Theorem \ref{thm:radialprops} where connectivity properties of the planar map are considered.

    A planar map $G$ is \emph{self-dual} if there is a map isomorphism between an embedding of $G$ and an embedding of $G^*$. The simplest example of a self-dual map is  $K_4$. Self-dual maps already appear in the works of Smith and Tutte \cite{smith1950}. Their first systematic study is due to Archdeacon and Richter, who classified simple 3-connected self-dual planar maps \cite{archdeacon1992construction}. Servatius and Servatius studied various forms of self-duality and classified self-dual pairings of planar graphs \cite{ServatiusServatius1994,ServatiusServatius1995,servatius1996selfdual}. 
    
    The idea to generate self-dual maps using radial graphs already appears in the work of Archdeacon and Richter \cite{archdeacon1992construction}. The idea is very natural: if the planar map is self-dual, then this self-duality gives a map automorphism of the radial graph (and also of the full barycentric subdivision). Hence we can construct self-dual graphs by finding quadrangulations that have a non-trivial automorphism. Their work focuses on 3-connected simple graphs. We extend one of their constructions and several of their observations to the 2-connected, not necessarily simple setting.  We make the correspondence explicit, and collect the elementary consequences that will be used repeatedly in the later papers of the series.

Suppose $G$ is a self-dual map. Following the correspondence given by $*$ and then the isomorphism of $G$ and $G^*$, we obtain a permutation $\tau:V\cup E\cup F\rightarrow V\cup E \cup F$, which we call a \emph{self-duality} of $G$. There might be multiple isomorphisms between $G$ and $G^*$, hence a self-dual map might have several self-dualities. We say that a self-duality  $\tau:G\rightarrow G^*$ is \emph{involutive} if $\tau^2=\id$. Solving a question of Grünbaum and Shepard, Jendrol showed that not all self-dual planar maps exhibit involutive self-dualities \cite{grunbaum1988,jendrol1989}.

We say that the self-duality  $\tau:V\cup E\cup F\rightarrow V\cup E \cup F$ is \emph{strongly involutive} if it satisfies the following conditions:
\begin{enumerate}
    \item[1)]  The self-duality is involutive, i.e., $\tau^2=\id$.
    \item[2)]  For every vertex $u\in V$ we have $u \notin \tau (u)$.
\end{enumerate}

Under the 2-connected hypotheses used below, condition $2)$ is equivalent to saying that the induced automorphism of the barycentric subdivision is fixed-point-free (see Lemma \ref{lem:sisdotherdef}).

Maps having a strongly involutive self-duality will be called \emph{SISD maps}, and graphs that have SISD maps will be called \emph{SISD graphs}. This family of graphs has attracted a lot of attention recently for its connection to several geometric questions. For example, they appear in the study of self-dual polytopes \cite{lovasz1983self,gouveia2023self,horvath2021strongly,horvath2025strongly}, in relation to Reuleaux polyhedrons \cite{montejano2020graphs,lopez2023borsuk}, and their general properties are studied in \cite{montejano2022self,bracho2020strongly,bracho202410}.

SISD maps were almost exclusively studied under the assumption that the graph is 3-connected. We extend some of these results to a more general setting. In \cite{bracho2020strongly} Bracho et al. showed that for a simple 3-connected SISD map there is only one strongly involutive self-duality. We extend this to the 2-connected and not necessarily simple case. 

\begin{theorem}\label{thm:onlyonestrong}
If $G$ is a 2-connected strongly involutive self-dual map, then there is only one self-duality which is strongly involutive. 
\end{theorem}

Secondly, we study simple non-bipartite quadrangulations of the projective plane, and we will call the corresponding map an \emph{NBQP map}. Graphs that can be embedded as an NBQP will be called \emph{NBQP graphs}. These graphs already appear in the works of Negami \cite{NE86,NE92}. They became well known after the famous result of Youngs \cite{YOU96} who showed that the chromatic number of an NBQP is always exactly 4. Since then, NBQP graphs have been extensively studied. We can find results on their chromatic properties \cite{mohar2006quadrangulations,nakamoto2004}, on their face-width \cite{esperet2018}, on their generating methods \cite{NEG93} and many more \cite{cairns2008,KAISER20151,MAT16}.  They are also closely related to \emph{parity embeddings}; see Section \ref{sec:applications}. 

The main goal of the paper is to show that these two classes of maps are in a natural one-to-one correspondence.  More precisely,  given a strongly involutive self-dual planar map $G$, one can take its radial map $R(G)$. This results in a planar quadrangulation that has an automorphism given by the strongly involutive self-duality of the original graph. As the self-duality is strongly involutive, this is a fixed-point-free automorphism and we can factor the map by this automorphism to obtain a quadrangulation $\overline R(G)$ of the projective plane; see Section \ref{sec:sisdgraphs} for details and Figure \ref{fig:sisdtonbqp} for an illustration. 

Conversely, consider the following construction due to Archdeacon and Richter \cite{archdeacon1992construction}. Given a simple non-bipartite quadrangulation $Q$ of $\Pb^2$, there is a natural diagonal map $\diag(Q)$, which is a planar strongly involutive self-dual map. To obtain $\diag(Q)$ we lift the quadrangulation $Q$ to the sphere and obtain a planar bipartite quadrangulation denoted by $\widetilde Q$. The map $\diag(Q)$ is constructed by taking all diagonals in $\widetilde Q$ corresponding to one of the color classes; see Section \ref{sec:quad} for details and Figure \ref{fig:diaggraph} for an illustration.

Our main result states that we have the following connections between these operations.

\begin{theorem}[SISD--NBQP correspondence]\label{thm:commute} The two operations, $\diag()$ and $\overline R()$, are inverse bijections between the family of simple non-bipartite quadrangulations of the projective plane  and the family of 2-connected strongly involutive self-dual maps. They satisfy the following properties.  
\begin{enumerate}
    \item[a)] If $G$ is a 2-connected strongly involutive self-dual map, then $G=\diag(\overline R(G))$ as planar maps.
    \item[b)] If $Q$ is a non-bipartite quadrangulation of the projective plane, then $Q=\overline R(\diag(Q))$ as planar maps.
    \item[c)] If $Q$ is a non-bipartite quadrangulation of the projective plane with cover map $\widetilde Q$, then $\widetilde Q= R(\diag(Q))$ as projective planar maps.
\end{enumerate}
    \end{theorem}

 See Figure \ref{fig:bigpic} for an illustration of the correspondence. We have also collected a large number of examples in  Appendix \ref{app:examples}. The operations are sensitive to the embeddings of the graphs; hence, the need to keep track of embeddings deserves emphasis. Several later geometric realizations depend not only on the abstract graph but also on the embedding of the graph. 

The main contribution of Part I is the clarification and packaging of known constructions into a precise dictionary between the non-bipartite quadrangulations of the projective plane and 2-connected strongly involutive self-dual maps, see Table \ref{tab:nbqp-sisd-dictionary}.

 \begin{table}[ht]
\centering
\renewcommand{\arraystretch}{1.25}
\begin{tabular}{c|c}
\textbf{NBQP side} & \textbf{SISD side} \\
\hline
Non-bipartite quadrangulation $Q$ of $\Pb^2$
&
Reduced radial map $\overline R(G)$
\\
Diagonal map  $\diag(Q)$
&
2-connected strongly involutive self-dual planar map $G$
\\

Canonical double cover map $\widetilde Q$
&
Radial map $R(G)$
\\

A face quadrangle of $Q$
&
A primal-dual edge pair $e,e^*$ of $G$
\\

The two diagonal classes of $\widetilde Q$
&
The primal and dual edge sets of $G$
\\

No separating $4$-cycle in $Q$
&
3--connectivity of $\diag(Q)$
\\

A separating $4$-cycle in $Q$
&
A $2$-cut in $\diag(Q)$
\\

\end{tabular}
\caption{Dictionary between non-bipartite quadrangulations of the projective plane and 2-connected strongly involutive self-dual planar maps.}
\label{tab:nbqp-sisd-dictionary}
\end{table}

Furthermore, we develop a number of useful lemmas on various properties of SISD maps and NBQP maps. For example, we have the following relations between planar maps and their radial maps. We say that a 4-cycle of $Q$ is a \emph{separating 4-cycle} if it is a 4-cycle that does not bound a face of $Q$.  

\begin{theorem}[{\cite[Theorem~5]{BrinkmannEtAl2005}}]\label{thm:radialprops}   
Let $G$ be a connected planar map, and let $R(G)$ be its radial map. 
\begin{enumerate}
    
    \item $R(G)$ is a simple quadrangulation if and only if $G$ is loop-free and $2$-connected (but not necessarily simple).
    \item $R(G)$ is a simple quadrangulation with minimum degree $3$ if and only if $G$ is loop-free and 2-connected (but not necessarily simple) with no vertices of degree 2 or faces of size 2;
    \item $R(G)$ is a (simple) 3-connected quadrangulation if and only if $G$ is simple, 2-connected, and 3-edge-connected;
    \item $R(G)$ is a (simple) 3-connected quadrangulation without separating 4-cycles if and only if $G$ is simple and 3-connected.
\end{enumerate}
\end{theorem}

In the same spirit, we show the following.  

\begin{theorem}\label{thm:3connected}
     For a non-bipartite quadrangulation of the projective plane $Q$, the following are equivalent:
     \begin{enumerate}
         \item[a)] $Q$ has no separating 4-cycle,
         \item[b)] $\diag(Q)$ is 3-connected,
         \item[c)] for every proper subset $S\subset V(Q)$ that contains at least 4 vertices, we have $|E(Q[S])|< 2|S|-2$.
     \end{enumerate}
\end{theorem}

Then,  by Theorem \ref{thm:commute}, we immediately obtain the following.
 
\begin{corollary}\label{col:3conequiv}
    Let $G$ be a 2-connected strongly involutive self-dual map. Then the following are equivalent:
    \begin{enumerate}
        \item $G$ is 3-connected.
        \item $\overline R(G)$ has no separating 4-cycles.
        \item For every proper subset $S\subset V(\overline R(G))$ that contains at least 4 vertices, we have $|E(\overline R(G)[S])|< 2|S|-2$.
    \end{enumerate}
    
\end{corollary}

The paper is organized as follows. In Section \ref{sec:notions} we introduce our basic notions. In Section \ref{sec:quad} we discuss non-bipartite quadrangulations of the projective plane and define the diagonal map. In Section \ref{sec:sisdgraphs} we study strongly involutive self-dual maps. We introduce the reduced radial map and show the NBQP-SISD correspondence, Theorem \ref{thm:commute}. We discuss connectivity properties of strongly involutive self-dual maps in Section \ref{sec:applications} and we also show a connection to parity embeddings.

\section{Embeddings, maps, and self-duality}\label{sec:notions}

A key idea that we explore in this series is self-duality. 
Several notions of self-duality for planar graphs appear in the literature, and they are not always equivalent; see \cite{servatius1996selfdual}. Since the distinction is important for the arguments in this paper, we begin by fixing our terminology carefully.

An embedding of a graph $G=(V,E)$ in a surface $\Sigma$ is a topological embedding $\varphi: G \hookrightarrow \Sigma$. An embedding divides the surface into regions called the \emph{faces} of the embedding. We denote the set of faces by $F$. An embedding is called a \emph{2-cell embedding} if every face is homeomorphic to an open disk. The \emph{size} of a face is the number of edges on its boundary, counted with multiplicity.

 Suppose $\varphi_1:G_1\hookrightarrow \Sigma$  and $\varphi_2:G_2\hookrightarrow \Sigma$ are embeddings of two graphs in the same surface. We say that $\varphi_1$ and $\varphi_2$ are  \emph{topologically isomorphic} if they differ only in a homeomorphism of the underlying surface; that is, if there is a homeomorphism $h:\Sigma\rightarrow \Sigma$ and a graph isomorphism $\gamma:G_1\rightarrow G_2$ such that $h\circ\varphi_1 =\varphi_2\circ\gamma$. Note that isomorphism of embeddings is not merely an incidence-preserving isomorphism of the incidence structure of the cells, but requires a homeomorphism of the underlying surface.

 A \emph{map}, in this paper, means an equivalence class of embeddings under topological isomorphism. That is, a map is not a particular embedding of a graph, but an embedded graph considered up to homeomorphism of the underlying surface. When a concrete embedding is needed, we call it an embedding of the map. Two maps are \emph{isomorphic} if they have topologically isomorphic embeddings. The notions defined for embeddings such as  vertices, edges, faces, etc. naturally extend to maps.  We use $\gr(M)$ to denote the graph of the map $M$.

For simple 3-connected graphs embedded in the plane, the definitions can be simplified. Indeed, by Whitney's theorem \cite{whitney19332}, any two embeddings of a simple 3-connected graph in the plane are topologically isomorphic. Thus the underlying graph already determines the map.  
However, we do not assume 3-connectivity or simplicity in this paper. For example, the same abstract graph may have non-isomorphic embeddings. Hence, all isomorphisms, automorphisms, and dualities used below are meant in the sense of topological isomorphism.

\subsection{On figures}
 We consider embeddings in the plane and in the projective plane. Mathematically we always consider planar embeddings as embeddings on a sphere, but the accompanying figures will be drawn in the plane. In some cases we would need to show centrally symmetric drawings on the sphere. Such drawings are rather hard to produce, so we trust the readers' imagination and show simple planar drawings instead.
 
 We depict projective planar embeddings in a disk, where opposite points of the boundary are identified; see, for example, Figure \ref{fig:q4} for a drawing of the complete graph on four vertices. If a figure has no dashed bounding circle around it, it is a planar drawing, or it represents a planar portion of the projective plane.

\section{Quadrangulations of the projective plane} \label{sec:quad}
 Let $\Sb^2$ denote the unit sphere in $\R^3$. We will think of the projective plane $\Pb^2$ as the topological space obtained by identifying antipodal pairs of points on $\Sb^2$. A \emph{cross-cap}, also known as a \emph{Möbius strip}, is the topological space created by deleting an open disk from $\Pb^2$. An \emph{open cross-cap} is a cross-cap minus its boundary. 
 
 In an embedding in the projective plane, each cycle of the graph is embedded as a closed curve. If this curve is not contractible, we say that the cycle is \emph{essential}.  

 A projective embedding divides the projective plane into \emph{faces}. Note that the interior of a face is not necessarily homeomorphic to an open disk, but if the graph is connected, then there are only two options: the interior of a face is either an open disk or an open cross-cap. It is also easy to see that if we have an embedding in $\Pb^2$ where each cycle of the graph is contractible, then the embedding is nothing more than a planar embedding where the interior of a cell is replaced with an open cross-cap. 

We say that a graph embedded on a surface $\Sigma$ is a \emph{quadrangulation} if it is a 2-cell embedding and each face has four distinct vertices and four edges on its boundary.  A quadrangulation $Q$ is \emph{simple} if it contains no loop edges and no parallel edges. Throughout this paper we consider simple quadrangulations. Note that in a simple quadrangulation we allow two quadrangles to share more than one edge, and it is also possible that a pair of vertices occurs as the edge of a quadrangle and at the same time as the diagonal of another one. See, for example, the quadrangulation of the projective plane in Figure \ref{fig:q4a}, which we denote by $Q_4$. In $Q_4$ every pair of vertices is an edge of a quadrangle and the diagonal of a different quadrangle. Another quadrangulation is $Q_6$, shown in Figure \ref{fig:q4b}.

\begin{figure}[!ht]
    \centering
     \begin{subfigure}[t]{0.2\textwidth}
    \centering
    \includegraphics[width=1.\linewidth]{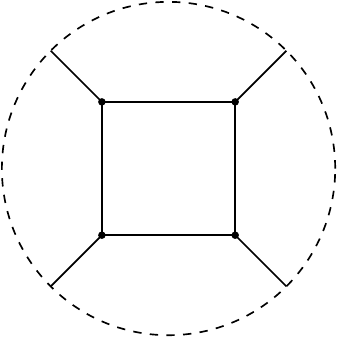}
    \caption{$Q_4$}
    \label{fig:q4a}
  \end{subfigure}
    \begin{subfigure}[t]{0.2\textwidth}
    \centering
    \includegraphics[width=1.\linewidth]{figures/q6_nocap.pdf}
    \caption{$Q_6$}
    \label{fig:q4b}
  \end{subfigure}

    \caption{The quadrangulations of $\Pb^2$ called $Q_4$ and $Q_6$. 
    }
    \label{fig:q4}
\end{figure}
 
 A \emph{quadrangulation map} is an equivalence class of topologically equivalent quadrangulations. We focus on simple non-bipartite quadrangulations of the projective plane, and we call such an embedding an \emph{NBQP embedding}.  We call the corresponding map an \emph{NBQP map}. Graphs that can be embedded in $\Pb^2$ as an NBQP will be called \emph{NBQP graphs}. 

    Note that an NBQP graph might have topologically non-equivalent NBQP embeddings; see Figure \ref{fig:isomdif} for an example. (On the other hand, we will see in Lemma \ref{lem:uniquestrduality} that graphs of NBQP maps without a separating 4-cycle admit a unique NBQP map.)

\begin{figure}[!ht]
    \centering
    \includegraphics[width=0.5\linewidth]{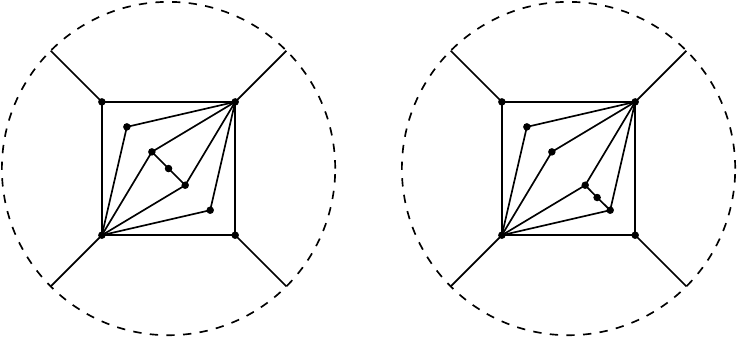}
    \caption{Two combinatorially different quadrangulations of $\Pb^2$ with isomorphic underlying graphs.}
    \label{fig:isomdif}
\end{figure}

The following lemma plays a central role in the study of quadrangulations.

\begin{lemma}[Negami and Nakamoto, Lemma 9 in \cite{NEG93}]\label{lem:cyclepar} Let $G$ be a graph embedded on a closed surface $\Sigma$ whose faces all have even size. If two cycles of $G$ have homotopic embeddings, then their lengths have the same parity.
\end{lemma}

\begin{corollary}\label{cor:nbqp-parity}
Let $Q$ be an NBQP map. A cycle of $Q$ is contractible if and only if it has even length. Equivalently, a cycle of $Q$ is essential in $\mathbb P^2$ if and only if it has odd length.
\end{corollary}
\begin{proof}
Since $Q$ is a 2-cell embedding, every facial $4$-cycle of $Q$ is contractible and even. By Lemma~\ref{lem:cyclepar}, every contractible cycle has even length. Since $Q$ is non-bipartite, it contains an odd cycle, and such a cycle cannot be contractible; hence it is essential. Finally, all essential simple closed curves in the projective plane are homotopic. Applying Lemma~\ref{lem:cyclepar} again, every essential cycle has the same parity as the odd one, and is therefore odd.
\end{proof}

Next, we derive an upper bound on the number of edges in a projective quadrangulation.

\begin{lemma}\label{lem:2nminus2isquad}
    Let $G=(V,E)$ be a simple graph embedded in the projective plane such that each cell has size at least 4. Then $|E|\le2|V|-2$ with equality if and only if $G$ is a quadrangulation. 
\end{lemma}
\begin{proof}
    First assume that $G$ is 2-connected. Let $F_i$ denote the $i$-th cell of the embedding. Let $\chi$ denote the Euler characteristic of a topological space. Using the generalized Euler formula we obtain that for a 2-connected graph embedded in $\Pb^2$ we have

    $$|V|-|E|+\sum_i\chi (F_i)=\chi(\Pb^2)=1.$$

    Since each $F_i$ is a disk or a cross-cap, we have $\chi(F_i)=1$ or $\chi(F_i)=0$ for each $i$. Hence,

    $$|V|-|E|+|F|\ge 1.$$

    Equality holds if and only if the embedding is a 2-cell embedding.   As $G$ has no cell of size less than 4, we have $4|F|\le 2|E|$ with equality if and only if each cell is bounded by 4 edges. Hence,

     $$ 4(1-|V|+|E|)\le 4|F|\le 2|E|.$$

    Therefore,
    
    $$|E|\le 2|V|-2.$$

    Equality holds if and only if each face is bounded by exactly 4 edges and the embedding is a 2-cell embedding, that is, if and only if $G$ is a quadrangulation.

    For graphs that are not necessarily 2-connected, the lemma follows from applying the same argument for each block of the graph.

\end{proof}

 Lemma \ref{lem:2nminus2isquad} implies a natural sparsity condition for quadrangulations; no subset of the vertices can span too many edges. For NBQP maps we make an even sharper observation.

\begin{lemma}\label{lem:maximal}
    Let $Q'=(V,E)$ be the subgraph of an NBQP graph $Q$. Then $Q'$ has $2|V|-2$ edges if and only if $Q'$ is an NBQP graph itself.
\end{lemma}
\begin{proof}
If $Q'$ is an NBQP graph, then by Lemma \ref{lem:2nminus2isquad} it has exactly $2|V|-2$ edges. 

In the other direction, suppose $Q'$ has $2|V|-2$ edges. Consider the embedding of $Q'$ inherited from the NBQP embedding of $Q$. Corollary \ref{cor:nbqp-parity} implies that a cycle of $Q$ is essential if and only if it has an odd length. This also means that $Q'$ cannot have a cell of size 3: walking along a cycle of length 3 we would end up on the opposite side of the starting edge. As $Q$ is simple we do not have loop edges or parallel edges, hence there are no cells of size 1 or 2.   

Hence, each cell of $Q'$ has size at least 4. By Lemma \ref{lem:2nminus2isquad} we obtain that $Q'$ (with the inherited embedding) is a quadrangulation of the projective plane. As it is a quadrangulation, it is a 2-cell embedding. Hence there must be a cycle in $Q'$ that is an essential cycle, otherwise the graph would lie in a planar part of $\Pb^2$, and the outer face would not be homeomorphic to a disk. Since $Q'$ inherited the embedding from $Q$, this cycle must be an odd cycle, and hence $Q'$ is non-bipartite. Therefore, $Q'$ is an NBQP graph.
\end{proof}

\subsection{The canonical double cover and the cover map}\label{subs:rotation}

 The \emph{canonical double cover}, denoted by  $H \times K_2$, of a graph $H=(V,E)$ is the graph where we split each vertex $i$ of $H$ into two copies $i_1,i_2$ and we replace each edge $ij$ of $H$ by two edges $i_1j_2$ and $i_2j_1$; see Figure \ref{fig:doublecover}. Throughout the paper, we will use $V_1$ and $V_2$ to denote the sets  $\{i_1\mid i\in V\}$ and $\{i_2\mid i\in V\}$, respectively. The canonical double cover of projective maps is a well-studied topic, see for example the works of Negami \cite{NE86,NE92}.

\begin{figure}[!h]
    \centering
    \includegraphics[width=0.7\linewidth]{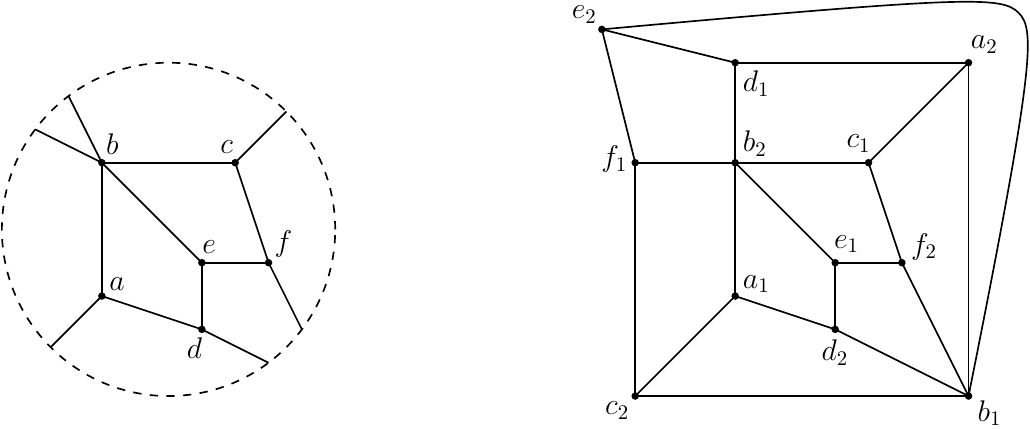}
    \caption{A quadrangulation of the projective plane and its canonical double cover on the plane.}
    \label{fig:doublecover}
\end{figure}

The canonical double cover has a number of nice properties. For example, it is always bipartite by construction. It is also clear that the number of vertices and the number of edges are doubled in the process.  $H\times K_2$ comes with a \emph{deck involution} $\iota$ which swaps $i_1$ with $i_2$ for each $i\in V$. The deck involution satisfies $v=\iota(\iota(v))$, $v$ and $\iota(v)$ are never neighbors, and $v\iota(w)$ is an edge if and only if $w\iota(v)$ is an edge.

Another nice property of $ H\times K_2$ is that it is connected if and only if $H$ is connected and not bipartite. If $H$ is bipartite, then  $H\times K_2$ is just the union of two disjoint copies of $H$. Furthermore, if $H$ is connected, then there is a unique 2-coloring of $H\times K_2$ such that each pair $i_1,i_2$ is colored differently, namely the coloring with color classes $V_1=\{i_1\mid i\in V\}$ and $V_2=\{i_2\mid i\in V\}$.

The sphere $\Sb^2$ covers the projective plane $\Pb^2$ via the quotient map $\pi:\Sb^2\to\Pb^2$ that identifies antipodal pairs. Hence, we can \emph{lift} maps from $\Pb^2$ to maps on $\Sb^2$. The resulting map has twice as many vertices and twice as many edges, and the new map is antipodally symmetric. For a projective map $M$, let $\widetilde M$ denote the lifted map on the sphere.

A graph $H'=(V',E')$ is a \emph{covering graph} of $H$ if there is a covering map from the vertex set of $H'$ to the vertex set of $H$. That is, a map $f:V'\rightarrow V$ such that for each $i\in V'$ the neighborhood $N_{H'}(i)$ is mapped bijectively onto the neighborhood $N_{H}(f(i))$ in $H$. For example, if $M$ is a projective map, then $\gr(\widetilde M)$ is a covering graph of $\gr(M)$. We simply map both $i_1$ and $i_2$ to $i$ for each $i\in V$.

When we lift a graph embedding from the projective plane to the sphere, we obtain a planar embedding of a covering graph. The resulting covering graph depends on the embedding of $H$ in the projective plane. For example, if $C_3$ is embedded as a contractible cycle, then it is lifted to two disjoint copies of $C_3$, but if it is embedded as an essential cycle, then it is lifted to a $C_6$. In the next lemma we show a very useful property of NBQP maps: if we lift them, we get the canonical double cover as the covering graph. The following lemma follows from a more general theorem of Negami, see \cite{NE92}. To keep this paper self-contained, we include an elementary proof for this case. 

\begin{lemma}[Special case of Theorem 4 in \cite{NE92}]\label{lem:bqlift}
Let $Q$ be an NBQP embedding. If we lift $Q$ to the sphere, we obtain a planar embedding of the canonical double cover of the graph of $Q$. That is, $\widetilde Q$ is a planar map of $\gr (Q)\times K_2$. 
\end{lemma}
\begin{proof}

    Let $H=\gr(\widetilde Q)$ be the cover graph of $Q$ that we obtain by lifting the embedding. Since $Q$ is connected and non-bipartite, there is an odd cycle $C$ in $Q$ and we fix a vertex $s$ of this cycle. Let $s_1,s_2\in V(H)$ be the two vertices covering $s$.  Since $C$ is an essential cycle, there is a walk in $H$ connecting $s_1$ and $s_2$. As $Q$ is connected, this implies that $H$ is also connected. When we lift the drawing of $Q$, each vertex $i$ is mapped to two vertices $i_1,i_2$ on the sphere. Any walk between $i_1$ and $i_2$ corresponds to an essential curve in the projective plane. By Lemma \ref{lem:cyclepar}, all these walks must have an odd length. Hence, by switching the roles of $i_1$ and $i_2$ if necessary, we can assume that for each $i$ any walk between $i_1$ and the fixed $s_1$ vertex has even length, and any walk between $i_2$ and $s_1$ has odd length. This immediately shows that $H$ is isomorphic to $ {\gr(Q)}\times K_2$, as each edge must be of the form $k_1l_2$ for some $k,l\in V(Q)$.

\end{proof}

\subsection{Diagonal map}
Suppose we have an NBQP map $Q=(V,E)$. Its \emph{diagonal graph} is the graph on vertex set $V$ where two vertices are joined with as many parallel edges as there are faces of $Q$ in which they form a diagonal. That is, if $ac$ appears as a diagonal in more than one face of $Q$, then there are parallel edges between $a$ and $c$. For each $Q$ we can create a drawing (but not an embedding) of the diagonal graph in the projective plane in the following way. We take an NBQP embedding of $Q$ and then draw each diagonal in the corresponding quadrangle (Figure \ref{fig:diagrapha}). In such a drawing the only intersections between edges are the ones between two diagonals of a given quadrangle.   

\begin{figure}[!ht]
    \centering
    \begin{subfigure}[c]{0.24\textwidth}
    \centering
    \includegraphics[width=0.99\linewidth]{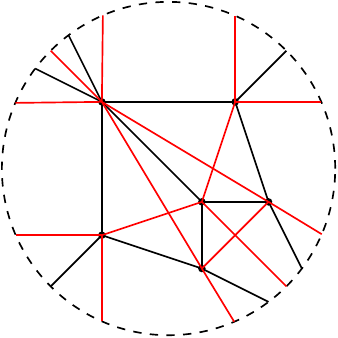}
    \caption{Diagonal graph drawn in $\Pb^2$.}
    \label{fig:diagrapha}
  \end{subfigure}
    \begin{subfigure}[c]{0.24\textwidth}
    \centering
    \includegraphics[width=0.99\linewidth]{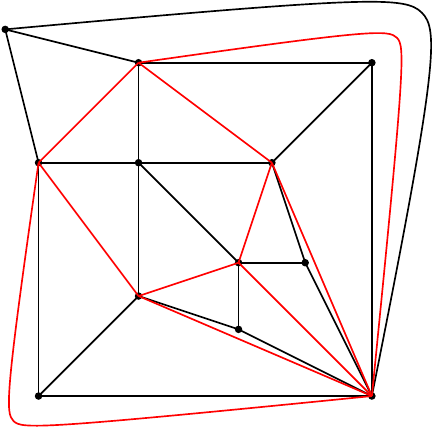}
    \caption{First copy of $\diag(Q)$ on $\widetilde Q$.}
    \label{fig:diagraphb}
  \end{subfigure}
    \begin{subfigure}[c]{0.24\textwidth}
    \centering
    \includegraphics[width=0.99\linewidth]{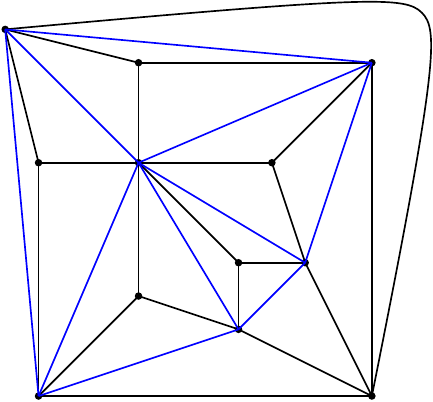}
    \caption{Second copy of $\diag(Q)$ on $\widetilde Q$.}
    \label{fig:diagraphc}
  \end{subfigure}
    \begin{subfigure}[c]{0.24\textwidth}
    \centering
    \includegraphics[width=0.99\linewidth]{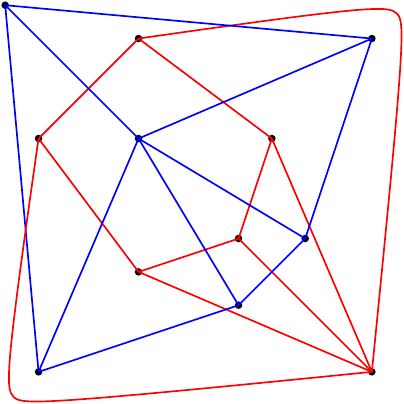}
    \caption{Self-dual embedding of the two copies of $\diag(Q)$.}
    \label{fig:diagraphd}
  \end{subfigure}
  
    \caption{Steps of creating a SISD map from an NBQP map.}
    \label{fig:diaggraph}
\end{figure}

We claim that the diagonal graph has a natural planar embedding. Indeed, by Lemma \ref{lem:bqlift} if we lift $Q$ to the sphere we get an embedding of ${\gr(Q)}\times K_2$. Each diagonal $ij$ is lifted to two diagonals, $i_1j_1$ and $i_2j_2$. Hence, if we lift the projective drawing of the diagonal graph to the sphere, we get two disjoint copies of the diagonal graph, one on vertex set $V_1$ and the other on vertex set $V_2$. In each quadrangle on the sphere one diagonal is in one of the copies and the other one is in the other. This implies that both copies of the diagonal graph are planar embeddings, and that they are antipodally symmetric. We define the \emph{diagonal map of $Q$}, denoted by $\diag(Q)$, as the planar map on $V_1$. Antipodally symmetric embeddings are topologically equivalent, hence the ambiguity of the lift does not cause any problems.

 The two copies of $\diag(Q)$ are clearly dual to each other; see Figure \ref{fig:diagraphd}. Therefore $\diag(Q)$ is a self-dual map on the sphere. This construction is actually one of the six constructions used by Archdeacon and Richter to generate all 3-connected simple self-dual planar graphs \cite{archdeacon1992construction}. We show that the self-duality is \emph{strongly involutive}, i.e., there is a self-duality $\tau$ that satisfies the following conditions: $\tau^2=\id$, and for every vertex $u$ we have $ u \notin \tau (u)$.
 
\begin{lemma}\label{lem:diagissisd}
    If $Q$ is an NBQP map, then the diagonal map $\diag(Q)$ is a SISD map.
\end{lemma}
\begin{proof}
Let $G_1$ and $G_2$ be the two copies of $\diag(Q)$ on vertex sets $V_1,V_2$, respectively. Let $\tau$ be the self-duality given by the antipodal map $\alpha(x)=-x$. Then $\tau$ maps $i_1$ to $i_2$ for all $i\in V$; on the lifted quadrangulation this is the deck involution of the double cover ${\gr (Q)}\times K_2 $. From the construction it is clear that $\alpha$ is indeed a self-duality homeomorphism; see Figure \ref{fig:diaggraph}. Furthermore, as $\alpha^2=\id$ we have $\tau^2=\id$. For lifted vertices $i_1,j_1\in V_1$, we have $i_1\in \tau(j_1)$ if and only if $i_1j_2$ is an edge of $\gr(Q)\times K_2$. As ${\gr(Q)}\times K_2$ is loop-free we have $i_1\notin \tau(i_1)$ for all $i\in V$. Therefore $\tau$ is strongly involutive.
\end{proof}

\begin{lemma}\label{lem:2connectivity}
If $Q$ is an NBQP map, then $\diag(Q)$ is 2-connected.     
\end{lemma}
\begin{proof}
   
    By definition $\diag(Q)$ is one of the diagonal graphs of the planar quadrangulation $\widetilde{Q}$. Since $Q$ is simple, so is $\widetilde{Q}$.  By Theorem \ref{thm:radialprops} the diagonal graph of $\widetilde Q$ is 2-connected.

\end{proof}

\begin{corollary}\label{cor:diagis2connsisd}
If $Q$ is an NBQP map, then $\diag(Q)$ is a $2$-connected SISD map.
\end{corollary}

\section{Strongly involutive self-dual graphs}\label{sec:sisdgraphs}

First, two simple but useful observations.  

\begin{lemma}\label{lem:selfdualloopfree}
    Any 2-connected self-dual planar graph is loop-free. 
\end{lemma}
\begin{proof}
    For a planar graph $G$, any loop edge corresponds to a cut-edge in the dual graph. If $G$ is self-dual, this implies that there is a cut edge in $G$ itself, and hence $G$ is not 2-connected.
\end{proof}

\begin{lemma}\label{lem:selfdualsimple}
    Any 3-connected self-dual planar graph is simple. 
\end{lemma}
\begin{proof}
    Let $G$ be a 3-connected self-dual planar map. By the previous lemma $G$ is loop-free. If we have parallel edges $e_1,e_2$, then together they form a closed curve $\gamma$. Consider the dual graph $G^*$. The dual edges $e_1^*$ and $e_2^*$ form an edge-cut of size 2 between the vertices of $G^*$  that are inside $\gamma$ and those outside $\gamma$. Hence the dual graph $G^*$ is not 3-connected. By self-duality $G$ is also not 3-connected, a contradiction. Hence we do not have parallel edges and $G$ is simple.
\end{proof}

Now we recall the barycentric subdivision of a planar map. Let $M$ be a planar map. The \emph{barycentric subdivision} $B(M)$ is the map obtained as follows. For every vertex $v\in V(M)$, every edge $e\in E(M)$, and every
face $f\in F(M)$, we introduce a vertex $b_v,b_e,b_f$, respectively. That is, the vertex set is $V(M)\cup E(M)\cup F(M)$. In the embedding of $B(M)$, the vertices $b_e$ are placed on the corresponding edges of $M$, and the vertices $b_f$ are placed inside the corresponding faces of $M$. We join two vertices for every incidence between the corresponding objects of the map. That is, $b_v$ is connected to $b_e$ if $v$ is an endpoint of $e$. The drawing of the $b_vb_e$ edge follows the $e$ edge. Note that if $e$ is a loop edge, we have two $b_vb_e$ edges. Similarly, we connect $b_f$ to every edge and vertex incident to $f$, possibly multiple times. Hence, the faces of $B(M)$  correspond to the flags $(v,e,f)$, where $v$ is an endpoint of $e$ and $e$ lies on the boundary of $f$. See Figure \ref{fig:bary} for an example.

\begin{figure}[!ht]
    \centering
    \includegraphics[width=0.95\linewidth]{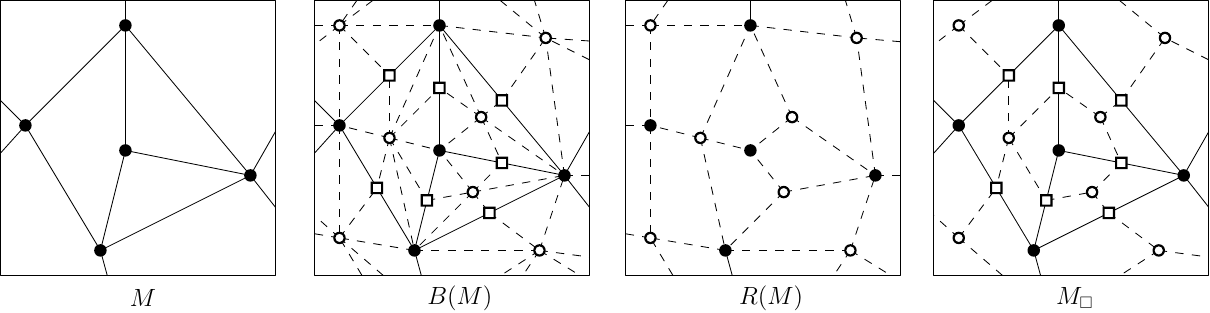}
    \caption{The barycentric subdivision, the radial map and the map of squares.}
    \label{fig:bary}
\end{figure}

We define two useful submaps of the barycentric subdivision. The \emph{map of squares of $M$} is the submap of the barycentric subdivision where we do not include the vertex-face edges, i.e. the edges of the form $b_vb_f$. The map of squares is denoted by $M_{\square}$.

The \emph{radial map of $M$} is the map given by the vertex-face edges, i.e. the induced subgraph on vertex set $V(M)\cup F(M)$. The radial map is denoted by $R(M)$. Each edge of the barycentric subdivision is either in the map of squares or in the radial map, see Figure \ref{fig:bary}.

\begin{lemma}\label{lem:barycentric3conn}
Let $M$ be a loop-free $2$-connected planar map. Then the barycentric subdivision $B(M)$ is a simple $3$-connected plane triangulation.
\end{lemma}
\begin{proof}
    By construction $B(M)$ is clearly loop-free. Suppose $v\in V(M),e\in E(M)$ and $f\in F(M)$. As $M$ is loop-free, there is at most one edge in $B(M)$ between $b_v$ and $b_e$. As $M$ is 2-connected, there are no cut edges. Hence there is at most one edge between $b_v$ and $b_f$ and at most one edge between $b_e$ and $b_f$. Therefore, $B(M)$ is a simple graph. $B(M)$ is a triangulation by construction, and every simple triangulation is 3-connected.
       
\end{proof}

Every self-duality $\tau$ of the map $M$ naturally induces a graph automorphism $\widehat{\tau}$ of $\gr(B(M))$: it sends the vertex $b_x$ corresponding to an object $x\in V(M)\cup E(M)\cup F(M)$ to $b_{\tau(x)}$. We will shortly see that this graph automorphism extends to a map automorphism $\overline{\tau}$. The second condition of being a strongly involutive self-duality is that $v\notin \tau(v)$ for every $v\in V$. Now we show that this condition corresponds to $\overline{\tau}$ being fixed-point-free. Variants of this idea already appeared in \cite{kupitz2010ball,bracho2020strongly,montejano2022self}, but again, we need a slightly more general version.

\begin{lemma}\label{lem:sisdotherdef}
    Suppose $G$ is a 2-connected self-dual planar map with involutive self-duality $\tau$, i.e., $\tau^2=\id$. Then the graph automorphism $\widehat{\tau}$ of $\gr(B(G))$ extends to a map automorphism $\overline{\tau}$ of $B(G)$. Furthermore, $\overline{\tau}$ is fixed-point-free if and only if $\tau$ is strongly involutive.
\end{lemma}
\begin{proof}
    By Lemma \ref{lem:selfdualloopfree} $G$ is loop-free, and hence, by Lemma \ref{lem:barycentric3conn}, $B(G)$ is a simple  3-connected triangulation. By Whitney's theorem \cite{whitney19332}, any two embeddings of a simple $3$-connected graph in the plane are topologically isomorphic. Hence $\widehat{\tau}$ extends to a map automorphism $\overline{\tau}$.
    
    Suppose $v\in \tau(v)$. Then $\overline{\tau}$ maps the edge $b_vb_{\tau(v)}$ to itself, and hence there is a fixed point. 

    In the other direction, suppose $x$ is a fixed point of $\overline{\tau}$. The fixed point is either a vertex of $B(G)$, or it lies in an edge or face of $B(G)$.
    
    If $x\in V(B(G))$, then it must be a vertex $b_{uv}$ corresponding to an edge $uv\in E(G)$, as the face-vertices and vertex-vertices are swapped. But if $b_{uv}$ is a fixed point, then $v$ is mapped to a face incident to $b_{uv}$, and hence $v\in \tau(v)$.

    If $x$ lies in an edge $b_vb_e$, then $\{b_v,b_e\}$ is set-wise fixed by $\overline{\tau}$. If $x$ lies in an edge $b_eb_f$, then $\{b_e,b_f\}$ is set-wise fixed by $\overline{\tau}$. In both cases $b_e$ must be a fixed point and we are done by the previous argument. 

    If $x$ lies in an edge $b_vb_f$, then $\{b_v,b_f\}$ is set-wise fixed by $\overline{\tau}$, and hence  $\overline{\tau}$ must swap them. This implies that $v\in f=\tau(v).$

    If $x$ lies in a face $b_vb_eb_f$, then $\{b_v,b_e,b_f\}$ must be set-wise fixed by $\overline{\tau}$. This implies that $b_e=\overline{\tau}(b_e)$, which implies that $v\in \tau(v)$, as we have seen. 

    Therefore, if there is a fixed point, we always find a vertex $v$ such that $v\in\tau(v)$ and the self-duality is not strongly involutive.

\end{proof}

In \cite{bracho2020strongly} Bracho et al. studied simple 3-connected SISD graphs. Although a planar map might have more than one self-duality (consider $K_4$), they showed that a simple 3-connected planar map cannot have two different strongly involutive self-dualities.

\begin{lemma}[Corollary 10 in \cite{bracho2020strongly}]\label{lem:dualityunique}
    If $G$ is a simple 3-connected strongly involutive self-dual map, then there is only one self-duality which is strongly involutive.
\end{lemma}

Here we extend this to the 2-connected case by proving Theorem \ref{thm:onlyonestrong}. We show that if $G$ is a 2-connected strongly involutive self-dual map, then there is only one self-duality which is strongly involutive.     

 We follow the same ideas as in \cite{bracho2020strongly}, slightly extending them. The main difference is that they relied on the map of squares $G_\square$, and used that it is 3-connected if $G$ is 3-connected. This is not necessarily true under the 2-connected assumption; hence,  we consider the full barycentric subdivision which is 3-connected by Lemma \ref{lem:barycentric3conn}. 
 
 Mani showed that every simple 3-connected planar graph has a representation as the skeleton of a polytope such that every automorphism of the graph extends to an isometry of the polytope \cite{Mani1971}. This also implies that every simple 3-connected planar graph has an embedding on $\Sb^2$ such that every automorphism extends to an isometry of $\Sb^2$. This was extended to the non-3-connected case by Servatius and Servatius in the following way.

\begin{lemma}[Lemma 1 in \cite{ServatiusServatius1995}]\label{lem:straight}
If $Aut(M)$ is the group of map automorphisms
of a connected map $M$, then the map can be redrawn so that $Aut(M)$ acts as a group of isometries on $\Sb^2$.
\end{lemma}

The idea is to draw $B(G)$ using this lemma and argue that self-duality must correspond to the antipodal isometry. 

\begin{proof}[Proof of Theorem \ref{thm:onlyonestrong}]

Let $G$ be a strongly involutive self-dual map. Let $\tau$ be a strongly involutive self-duality of $G$. That is, $\tau^2=\id$, and $v\notin\tau(v)$ for every $v$. Let $\widehat{\tau}$ be the graph automorphism of $\gr(B(G))$ given by $\tau$. As $\gr(B(G))$ is simple 3-connected (Lemma \ref{lem:barycentric3conn}), any graph automorphism extends to a map automorphism of $B(G)$. Hence $\widehat{\tau}$ extends to a map automorphism $\overline{\tau}$. By Lemma \ref{lem:straight} there is an embedding of $B(G)$ such that $Aut(B(G))$ acts as a group of isometries on $\Sb^2$.

Therefore, there is an isometry $\alpha: \Sb^2\rightarrow \Sb^2$ corresponding to $\overline{\tau}$ in such an embedding of $B(G)$. We show that this isometry is  the antipodal map.  We first observe that since $\tau^2=\id$, $\alpha^2$ is an isometry of the sphere that leaves $B(G)$, a simple 3-connected map, fixed. Therefore $\alpha^2=\id$. Secondly, by Lemma \ref{lem:sisdotherdef}, $\alpha$ is fixed-point-free. The only involutive isometry of the sphere that is fixed-point-free is the antipodal map, hence $\alpha(x)=-x$.

Suppose now that $\tau_1$ and $\tau_2$ are two strongly involutive self-dualities of $G$. Using Lemma \ref{lem:straight}, take an embedding of $B(G)$ such that both $\widehat{\tau}_1$ and $\widehat{\tau}_2$ extend to isometries of the sphere. By the previous argument they both extend to the antipodal map. Therefore $\widehat{\tau}_1=\widehat{\tau}_2$, implying that $\tau_1$ and $\tau_2$  are the same self-duality.

\end{proof}

\subsection{Reduced radial map and the SISD--NBQP correspondence}

As we have seen (Lemmas \ref{lem:diagissisd} and \ref{lem:2connectivity}), we can create a 2-connected SISD map from each NBQP map. Now we argue that this also works the other way. Not too surprisingly, we will use the same ideas in reverse. Just as we have drawn two copies of the diagonal graph in an antipodally symmetric way, we will argue that for any 2-connected SISD map we can embed two copies of it on the sphere such that they are antipodally symmetric and dual to each other. 

Let $G$ be a self-dual map and suppose that $G$ and $G^*$ are embedded in the sphere such that  $G^*$ is a dual embedding of $G$ and the following also holds for any $x \in \mathbb{S}^2$:

\begin{itemize}
    \item if $x \in V(G)$ then $-x \in V (G^*)$ and 
    \item if $x \in e \in E(G)$ then $-x \in e^{*} \in E(G^*)$, that is, $e^*$ is antipodally embedded in $\mathbb{S}^2$ with respect to the embedding of $e$.
\end{itemize}

Then we say that the embedding is \emph{antipodally self-dual}. Montejano, Alfonsín and Rasskin showed that strongly involutive graphs have antipodally self-dual embeddings \cite{montejano2022self}. Note that they work with a slightly different notion of strongly involutive self-duality. Instead of the condition $\tau^2=\id$ they require a weaker condition, namely that $v\in  \tau(w)$ if and only if $w\in\tau(v)$ for every $v,w\in V$. These two conditions are equivalent for simple 3-connected  graphs. Their proof starts with the claim that they are also equivalent for general, not necessarily 3-connected or simple graphs. We could not verify this claim\footnote{Indeed, consider the graph on vertex set $\{a,b,c\}$ and edges $ab,ac,bc,bc$, that is, a triangle with one of the edges doubled. This graph does have a self-duality satisfying $v\in \tau(w)$ if and only if $w\in\tau(v)$ but not satisfying $\tau^2=\id$. So the two conditions are not equivalent in general. On the other hand, the self-duality $\tau$ in this example does have a vertex $v$ such that $v\in\tau(v)$. Hence it could be that their two conditions together are equivalent to our two conditions together.}. We decided to give a proof of the theorem under our definition, mostly following their ideas.

\begin{theorem}[2-connected version of Corollary 2 in \cite{montejano2022self}]\label{thm:strisantipodal}
If $G$ is a 2-connected planar map with a strongly involutive self-duality $\tau$, then there is an antipodally self-dual embedding of $G$ and $G^*$ with respect to $\tau$. 
\end{theorem}

\begin{proof}[Proof of Theorem \ref{thm:strisantipodal}]    
The proof of Theorem \ref{thm:strisantipodal} is essentially the same idea as the one we used in Theorem \ref{thm:onlyonestrong}. Let $\overline{\tau}$ be the map automorphism of $B(G)$ given by $\tau$. Use Lemma \ref{lem:straight} to embed $B(G)$ in a way that $Aut(B(G))$ acts as a group of isometries of the sphere. Then the isometry corresponding to $\overline{\tau}$ is the antipodal map. Hence we can embed $G$ and $G^*$ in an antipodally symmetric way by following the edges of $B(G)$.   
\end{proof}

As we have seen, if $Q$ is an NBQP map, then $\diag(Q)$ is a 2-connected SISD map.  Theorem \ref{thm:strisantipodal} allows us to define the inverse operation. 
The planar map $R(G)$ will serve as the canonical double cover of the NBQP we are looking for. In the following lemma we  show that $R(G)$ is indeed the canonical double cover of a graph, and we can collapse it to obtain an NBQP.

\begin{lemma}\label{lem:sisdtoquad}
    If $G=(V,E,F)$ is a 2-connected SISD map, then the planar map $R(G)$ can be embedded on $\Sb^2$ as an antipodally symmetric quadrangulation. Furthermore, if we collapse the drawing to the projective plane, then we obtain an NBQP embedding. The resulting NBQP map depends only on the map $G$.
\end{lemma}

\begin{proof}
 As $G$ is a 2-connected SISD map, by Theorem \ref{thm:onlyonestrong} there is a unique self-duality $\tau_G$ of $G$ that is strongly involutive.  By Theorem \ref{thm:strisantipodal} there is an embedding of $G$ and $G^*$ that is antipodally self-dual with respect to $\tau_G$. Then $G$ and $G^*$ together divide the sphere into \emph{kites}, i.e. quadrangles such that each has one vertex from $V$, one from $V^*$ and two corresponding to the intersections of dual edge pairs, i.e. an intersection of the form $e\cap e^*$, see Figure \ref{fig:sisdtonbqpa}. 
 
 We create a new drawing with vertex set $V\cup V^*$. In each kite we connect the vertex from $V$ to the vertex from $V^*$ by an edge, and we do it symmetrically in opposite kites. This way we get an antipodally symmetric embedding of $R(G)$, see Figure \ref{fig:sisdtonbqpb}. The faces of this new drawing correspond to dual edge pairs. As $G$ is 2-connected and self-dual, it is loop-free (Lemma \ref{lem:selfdualloopfree}). By Theorem \ref{thm:radialprops}, $R(G)$ is a simple quadrangulation. In particular, it is a 2-cell embedding.

\begin{figure}[!htbp]
  \centering

  \begin{subfigure}[c]{0.3\textwidth}
    \centering
    \includegraphics[width=0.75\linewidth]{figures/self_dual.pdf}
    \caption{Self-dual embedding on the plane, dividing the plane into kites.}
    \label{fig:sisdtonbqpa}
  \end{subfigure}
  \begin{subfigure}[c]{0.35\textwidth}
    \centering
    \includegraphics[width=0.75\linewidth]{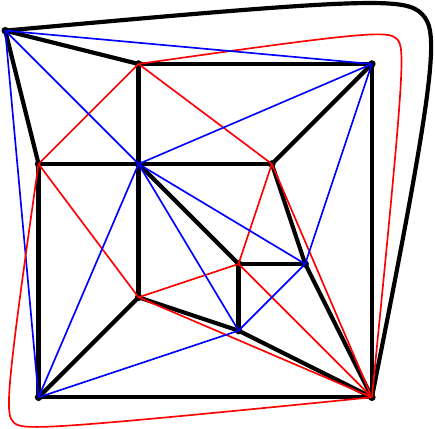}
    \caption{The drawing of $R(G)$.}
    \label{fig:sisdtonbqpb}
  \end{subfigure}
  \begin{subfigure}[c]{0.3\textwidth}
    \centering
    \includegraphics[width=0.75\linewidth]{figures/q6_nocap.pdf}
    \caption{The collapsed quadrangulation $\overline{R}(G)$ in $\Pb^2$.}
    \label{fig:sisdtonbqpc}
  \end{subfigure}

  \caption{Steps of creating an NBQP map from a SISD map.}
  \label{fig:sisdtonbqp}
\end{figure}

 Collapsing $R(G)$ to $\Pb^2$ results in a 2-cell embedding  $\overline R(G)$. Our goal is to show that this is an NBQP embedding. That is, we need to show that the collapsed drawing is indeed a quadrangulation, that it is simple, and also that it is not bipartite. 
 
 Consider a quadrangle $abcd$ in the drawing of $R(G)$, where $a,c\in V$ and $b,d\in V^*$. As we have started with an antipodally self-dual embedding of $G$ and $G^*$, $a$ and $c$ cannot be antipodally symmetric, and similarly $b$ and $d$ cannot be antipodally symmetric.   
  Also, as $v\notin \tau_G(v)$ for any $v\in V $, none of the $(a,b), (a,d),(b,c),(d,c)$ pairs can be antipodally symmetric. Hence, collapsing $abcd$ results in four distinct vertices, i.e. $\overline R(G)$ is a quadrangulation.  
  
  It is also not hard to see that $\overline R(G)$ is simple. A loop edge at $v$ would mean that $v\in \tau(v)$ which is forbidden. As $R(G)$ is simple, parallel edges in $\overline R(G)$ would mean that a vertex in $R(G)$ is connected to antipodally symmetric points of the sphere. This is not possible, as antipodally symmetric points cannot be both in $G$ or both in $G^*$.
 
 Clearly, opposite vertices have different colors in the 2-coloring of $R(G)$. Hence, if we collapse the quadrangulation to the projective plane, the result cannot be bipartite. Therefore, we have an NBQP embedding of $\overline R(G)$. Since $\tau_G$ is the unique strongly involutive self-duality of $G$, the map $\overline R(G)$ is determined by the map of $G$.
    
\end{proof}

For a 2-connected SISD map $G$, let $\tau_G$ denote its unique strongly involutive self-duality. Then $\tau_G$ extends to a fixed-point-free map automorphism $\overline{\tau_G}$ of $R(G)$.  We define\[\overline R(G):=R(G)/\overline{\tau_G}.\] Equivalently, $\overline R(G)$ is the NBQP map obtained by collapsing the antipodally symmetric drawing of $R(G)$ from the previous lemma. We call $\overline R(G)$ the \emph{reduced radial map} of $G$, and the lemma shows that this is well-defined\footnote{There is a similar notion in incidence geometry. Given a set of points and abstract lines, the Levi graph is the bipartite graph that has a vertex for each point and each line, and edges correspond to incidences. If the configuration is self-dual and fixed-point-free, then the quotient of the Levi graph by the self-duality is called the reduced Levi graph \cite{artzy1956}.}.

Now we are ready to prove Theorem \ref{thm:commute}. That is, we need to show the following. The two operations, $\diag()$ and $\overline R()$, are inverse bijections between the family of NBQP maps and the family of 2-connected SISD maps and satisfy the following properties.  

\begin{enumerate}
    \item[a)] If $G$ is a 2-connected SISD map, then $\diag(\overline R(G))$ is the same planar map as $G$.
    \item[b)] If $Q$ is an NBQP map, then $\overline R(\diag(Q))$ is the same NBQP map as $Q$.
    \item[c)] If $Q$ is an NBQP map, then $\widetilde Q=R(\diag(Q))$. 
\end{enumerate}

\begin{proof}
    We quickly recall the constructions. We create $\overline R(G)$ in three steps; see the first row of Figure \ref{fig:bigpic}:
    \begin{enumerate}
    \setlength\itemsep{0pt}
    \setlength\topsep{0pt}
    \setlength\parsep{0pt}
    \setlength\partopsep{0pt}
        \item[$i_1)$] Take an antipodally symmetric embedding of $G$ and $G^*$.
        \item[$i_2)$] Draw the appropriate diagonal in each kite to form $R(G)$.
        \item[$i_3)$] Collapse the embedding to $\Pb^2$.
    \end{enumerate}

    We construct $\diag(Q)$ in three steps, see the second row of Figure \ref{fig:bigpic}:
    \begin{enumerate}
    \setlength\itemsep{0pt}
    \setlength\topsep{0pt}
    \setlength\parsep{0pt}
    \setlength\partopsep{0pt}
        \item[$d_1)$] Lift $Q$ to $\Sb^2$ to obtain $\widetilde Q$.
        \item[$d_2)$] Take a 2-coloring of $\widetilde Q$.
        \item[$d_3)$] Draw the diagonal in each quadrangle corresponding to one of the color classes.
    \end{enumerate}

\begin{figure}[!ht]
    \centering
      \begin{subfigure}[b]{0.2\textwidth}
    \centering
    \includegraphics[width=1.\linewidth]{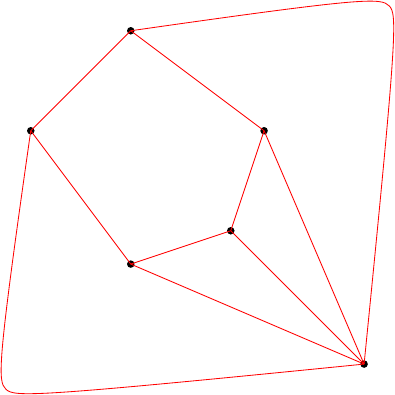}
    \caption{$G$}
    \label{fig:tba}
      \end{subfigure}
        \begin{subfigure}[b]{0.2\textwidth}
    \centering
    \includegraphics[width=1.\linewidth]{figures/self_dual.pdf}
    \caption{$G\cup G^*$}
    \label{fig:tbb}
      \end{subfigure}
   \begin{subfigure}[b]{0.2\textwidth}
    \centering
    \includegraphics[width=1.\linewidth]{figures/defIg.pdf}
    \caption{$R(G)$}
    \label{fig:tbc}
      \end{subfigure}
  \begin{subfigure}[b]{0.2\textwidth}
    \centering
    \includegraphics[width=1.\linewidth]{figures/q6_nocap.pdf}
    \caption{$\overline R(G)$}
    \label{fig:tbd}
      \end{subfigure}

      \begin{subfigure}[b]{0.2\textwidth}
    \centering
    \includegraphics[width=1.\linewidth]{figures/q6_nocap.pdf}
    \caption{$Q$}
    \label{fig:tbe}
      \end{subfigure}
        \begin{subfigure}[b]{0.2\textwidth}
    \centering
    \includegraphics[width=1.\linewidth]{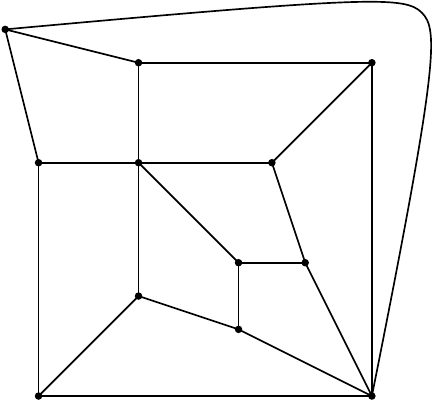}
    \caption{$\widetilde Q$}
    \label{fig:tbf}
      \end{subfigure}
   \begin{subfigure}[b]{0.2\textwidth}
    \centering
    \includegraphics[width=1.\linewidth]{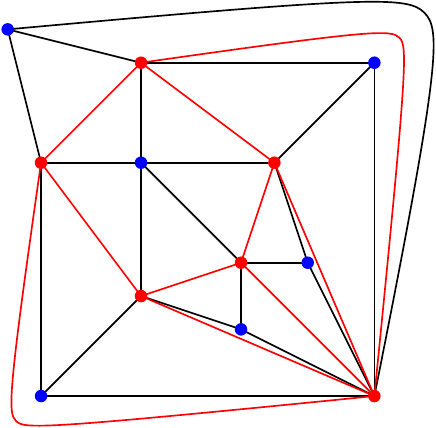}
    \caption{2-coloring}
    \label{fig:tbg}
      \end{subfigure}
  \begin{subfigure}[b]{0.2\textwidth}
    \centering
    \includegraphics[width=1.\linewidth]{figures/diagonly.pdf}
    \caption{$\diag(Q)$}
    \label{fig:tbh}
      \end{subfigure}

    \caption{From SISD maps to NBQP maps, and back.}
    \label{fig:bigpic}
\end{figure}

    First, we start with a $2$-connected SISD map $G$, and we take an antipodally self-dual embedding using Theorem \ref{thm:strisantipodal}. By definition, $\overline R(G)$ is obtained by quotienting the antipodally symmetric radial map $R(G)$ by the antipodal involution. 
    The two bipartition classes of $R(G)$ are the vertices and the faces of $G$. In every quadrangular face of $R(G)$, one diagonal is an edge of $G$ and the other is the corresponding dual edge of $G^*$. Thus drawing one diagonal class in the lift recovers $G$, and drawing the other recovers $G^*$. Since the SISD map identifies $G$ and $G^*$, we obtain $\diag(\overline R(G))=G$ as planar maps, hence part $a)$ holds.

Conversely, start with an NBQP map $Q$. Lifting $Q$ to the sphere gives $\widetilde Q$. The two diagonal classes of $\widetilde Q$ are dual planar maps; choosing one of them is precisely the construction of $\diag(Q)$. The radial map of $\diag(Q)$ is naturally $\widetilde Q$: vertices of $\diag(Q)$ correspond to one color class of $\widetilde Q$, faces of $\diag(Q)$ correspond to the other color class, and vertex-face incidences correspond exactly to edges of $\widetilde Q$. This shows part $c)$. Quotienting the radial map by the antipodal involution therefore recovers the original projective map $Q$.  By Theorem \ref{thm:onlyonestrong}  the antipodal map is the only strongly involutive self-duality. Hence $\overline R(\diag(Q))= Q$, and part $b)$ holds.

 Finally, to show that $\diag()$ and $\overline{R}()$ are bijections between the NBQP maps and 2-connected SISD maps and that they are inverses, it is enough to show that they map from one family to the other and that their composition is the identity. The mapping part is exactly  Lemma \ref{lem:diagissisd}, Corollary \ref{cor:diagis2connsisd}, and Lemma \ref{lem:sisdtoquad}. The composition part is $a)$ and $b)$ of this theorem.

\end{proof}

\section{Further properties}\label{sec:applications}
As mentioned before, SISD maps were mostly studied under the assumption that the graph is 3-connected. We start this section by discussing how 3-connectivity can be recognized on the NBQP side of the correspondence. Again, the results we obtain are analogous to the planar-radial setting. Consider, for example, the following theorem. Recall that $abcd$ is a \emph{separating 4-cycle} of $Q$ if it is a 4-cycle that does not bound a face of $Q$. 

\begin{theorem}[Theorem 2.8.2 in \cite{mohar2001graphs}]\label{thm:radial3con}
    Let $G$ be a simple 2-connected planar embedding on at least 4 vertices. Then the following are equivalent:
    \begin{enumerate}
        \item $G$ is 3-connected. 
        \item $R(G)$ has no separating 4-cycles.
        \item For every proper subset $S\subset V(R(G))$ that contains at least 5 vertices we have $|E(R(G)[S])|<2|S|- 4$.
    \end{enumerate}
\end{theorem}

We wish to show a similar result for strongly involutive self-dual graphs. Note that we do not assume that our graphs are simple. We start on the NBQP side and then we use Theorem \ref{thm:commute} to go back to the SISD world. Suppose $Q$ is an NBQP map. As we have seen $\diag(Q)$ is always 2-connected. Even if we assume that $Q$ is 3-connected $\diag(Q)$ might not be, see Figure \ref{fig:daignotconnected}. On the other hand, it is possible to describe the cases when $\diag(Q)$ is 3-connected using separating cycles.

\begin{figure}[!h]
    \centering

  \begin{subfigure}[c]{0.3\textwidth}
    \centering
    \includegraphics[width=0.75\linewidth]{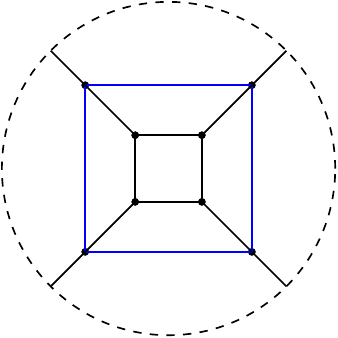}
    \caption{A 3-connected NBQP $Q$ with a separating 4-cycle.}
    \label{fig:daignotconnecteda}
  \end{subfigure}
  \begin{subfigure}[c]{0.3\textwidth}
    \centering
    \includegraphics[width=0.75\linewidth]{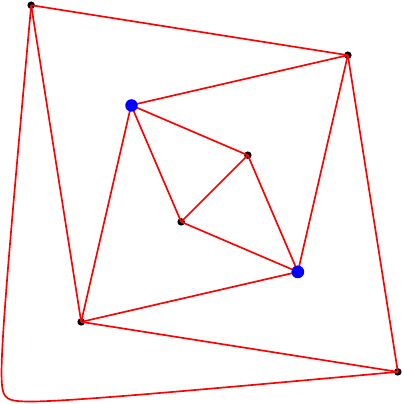}
    \caption{The diagonal graph of $Q$ is not 3-connected.}
    \label{fig:daignotconnectedb}
  \end{subfigure}
  \begin{subfigure}[c]{0.3\textwidth}
    \centering
    \includegraphics[width=0.75\linewidth]{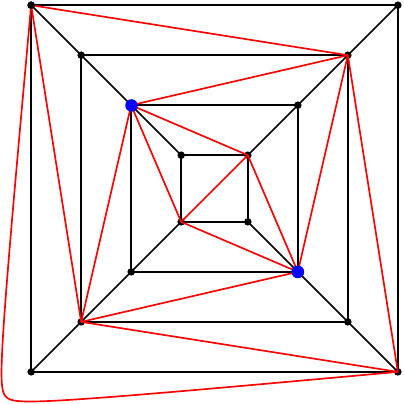}
    \caption{The diagonal graph on the double cover.}
    \label{fig:daignotconnectedc}
  \end{subfigure}

\caption{A 3-connected quadrangulation with a non-3-connected diagonal graph.}
    \label{fig:daignotconnected}
\end{figure}

Now we show Theorem \ref{thm:3connected}. That is, we show that for an NBQP map $Q$ the following are equivalent:
     \begin{enumerate}
         \item[a)] $Q$ has no separating 4-cycle,
         \item[b)] $\diag(Q)$ is 3-connected,
         \item[c)] for every proper subset $S\subset V$ that contains at least 4 vertices, we have $|E(Q[S])|< 2|S|-2$.
     \end{enumerate}

\begin{proof}[Proof of Theorem \ref{thm:3connected}]
   $b)\Rightarrow a)$  By definition $\diag(Q)$ is one of the diagonal graphs of the cover $\widetilde Q$. Suppose $\diag(Q)$ is 3-connected. As it is self-dual and 3-connected, it is simple (Lemma \ref{lem:selfdualsimple}). Hence, by Theorem \ref{thm:radialprops}, $\widetilde{Q}$ is separating 4-cycle free. This implies that $Q$ is separating 4-cycle free.

    $a)\Rightarrow b)$ In the other direction, let $G_1$ and $G_2$ be the two lifted copies of $\diag(Q)$ on vertex sets $V_1=\{i_1\mid i\in V(Q)\}$ and $V_2=\{i_2\mid i\in V(Q)\}$, respectively. Suppose there is a 2-cut of $G_1$ given by vertices $i_1,j_1$. Let $V'\subset V_1$ denote the vertices in one of the connected components after deleting $i_1,j_1$ from $G_1$. Consider the set $\mathcal A$ of quadrangles that have a diagonal spanned by $V'\cup\{i_1,j_1\}$, but the diagonal is not $i_1j_1$. The set $\mathcal{A}$ is non-empty as it contains all quadrangles incident to any vertex $v\in V'$. As $\widetilde Q$ is bipartite, with color classes $V_1,V_2$, on the boundary of $\cup \mathcal A$ the vertices alternate between $V_1$ and $V_2$. On the other hand, $i_1,j_1$ are the only possible vertices on the boundary from $V_1$, otherwise we would have a diagonal in $G_1$ leaving the connected component $V'$. As $V'$ does not contain all the vertices of $V_1$, the boundary of $\cup\mathcal A$ is not empty. Hence, it contains $i_1$ or $j_1$ and these are the only vertices of $V_1$ that are on the boundary. As $\widetilde Q$ is simple, any vertex of the boundary that is from $V_2$ must be connected to two different vertices from $V_1$. These can only be $i_1$ and $j_1$. Hence, $\cup \mathcal A$ can be decomposed into simply connected parts $P_1,\dots, P_k$ such that each $P_i$ is the union of some quadrangles of $\widetilde Q$, each is bounded by a 4-cycle of the form $a_2i_1b_2j_1$, and the interiors of the sets $P_i$ are disjoint. None of these parts can be a single quadrangle of $\widetilde Q$ as we excluded quadrangles from $\mathcal A$ that have $i_1j_1$ as a diagonal. Hence, each $P_i$ contains at least one vertex of $V'$ in its interior. As each $P_i$ is the union of quadrangles, no diagonal runs between vertices in the interior of different $P_i$'s. As $V'$ spans a connected subgraph of $\diag(Q)$, this implies that we cannot have more than one part. Hence, $\cup\mathcal A=P_1$, i.e. the boundary is a quadrangle, which corresponds to a non-facial 4-cycle in $\widetilde Q$ that contains exactly the vertices in $V'$. Projecting back to $\Pb^2$ gives a separating cycle of $Q$.

    $c)\Rightarrow a)$ Assume there is a separating 4-cycle $C$. By Lemma \ref{lem:cyclepar}, $C$ must be contractible. Hence it cuts $\Pb^2$ into a disk and a cross-cap. We can take $S$ to be the vertices of $C$ together with all the vertices in the cross-cap part. Let $Q_2$ be the quadrangulation of $\Pb^2$ given by the vertices $S$ and the edges of $Q$ that run in the cross-cap part. By Lemma \ref{lem:2nminus2isquad}  $Q_2$  has $2|S|-2$ edges, and hence $Q[S]$ has at least $2|S|-2$ edges.

    $a)\Rightarrow c)$  By Lemma \ref{lem:2nminus2isquad} we have $|E(Q[S])|\le 2|S|-2$. Hence, if the inequality fails for an $S$, then we must have $|E(Q[S])|= 2|S|-2$. Lemma \ref{lem:maximal} implies that $Q[S]$, with the inherited embedding, is an NBQP. Hence, the face of $Q[S]$ containing some vertex of $V(Q)\setminus S$ gives a separating 4-cycle. 
    
\end{proof}

Combining Theorem \ref{thm:3connected} and Theorem \ref{thm:commute} gives Corollary \ref{col:3conequiv}. That is, if $G$ is a 2-connected SISD map, then the following are equivalent:
    \begin{enumerate}
        \item $G$ is 3-connected.
        \item $\overline R(G)$ has no separating 4-cycle.
        \item For every proper subset $S\subset V(\overline R(G))$ that contains at least 4 vertices, we have $|E(\overline R(G)[S])|< 2|S|-2$.
    \end{enumerate}

This is exactly the desired analog of Theorem \ref{thm:radial3con}. Next we discuss when an NBQP graph has a unique NBQP embedding. By Whitney’s theorem, every simple 3-connected planar graph has a unique planar map, i.e. all embeddings are topologically equivalent \cite{whitney19332}. A nice analog of this result is due to Negami \cite{NE92}.

\begin{theorem}[Theorem 2 in \cite{NE92}]
Let $G$ be a 3-connected graph which is projective-planar but is not planar. If $G$ can be embedded in the projective plane so that it has only even faces, then such an embedding is unique. 
\end{theorem}

We show a variant of this for NBQP maps.

\begin{lemma}    \label{lem:uniquestrduality}
     If an NBQP map $Q$ has no separating 4-cycle, then all NBQP embeddings of the graph of $Q$ are topologically equivalent.
\end{lemma}
\begin{proof}
    Suppose $Q$ is an NBQP embedding without separating 4-cycles and $Q'$ is another NBQP embedding of $\gr(Q)$, the graph of $Q$. As $Q$ has no separating 4-cycles, the set of 4-cycles in $\gr(Q)$ is in bijection with the $|V|-1$ faces of $Q$. In $Q'$ we also have exactly $|V|-1$ faces, hence the 4-cycles of $\gr(Q)$ must also be in bijection with the faces of $Q'$. Hence $Q'$ is also separating 4-cycle free.  Furthermore,  $\gr(\diag(Q))$ and $\gr(\diag(Q'))$ are isomorphic graphs, both can be constructed by taking diagonals in each 4-cycle of $\gr(Q)$. 
    
     By Theorem \ref{thm:3connected} we know that $\gr(\diag(Q))$ and  $\gr(\diag(Q'))$ are 3-connected. By Lemma \ref{lem:selfdualsimple} they are also simple. By Whitney's theorem they have a unique planar embedding, hence $\diag(Q)$ and $\diag(Q')$ are the same planar map. By Theorem \ref{thm:commute} this implies that $Q$ and $Q'$ are topologically equivalent.

\end{proof}

This also implies that if we have an NBQP graph, then either all NBQP maps of it have separating 4-cycles, or none of them have.

\subsection{A connection to parity embeddings}

 A simple graph embedded in the projective plane is a \emph{parity embedding} when a simple cycle is essential if and only if it is of odd length. By Lemma \ref{lem:cyclepar}  submaps of NBQP maps are parity embeddings. We show that the reverse also holds. We consider this a folklore result. We will later need the following form, and the author could not find  this exact form in the literature, hence we include a proof. The idea is obvious: if we have a parity embedding, we keep adding edges (and vertices) until it becomes a NBQP map.

\begin{lemma}[Folklore]\label{lem:eqmaxcases}
    A simple graph $Q$ has a parity embedding in the projective plane if and only if it is a subgraph of an NBQP graph.
\end{lemma}

\begin{proof}
    Suppose $Q$ is a subgraph of an NBQP map. It is enough to consider the case where $Q$ itself is an NBQP embedding. The boundary of a quadrangle is contractible, and hence all contractible cycles must be even cycles by Lemma \ref{lem:cyclepar}. As an NBQP graph is not bipartite, we must have an odd cycle, and it must be essential. Using Lemma \ref{lem:cyclepar} again, all essential cycles must have an odd length.   

    In the other direction, assume that $Q$ is a parity embedding in the projective plane. If $Q$ is an NBQP embedding, we are done. Otherwise, we will add edges until we reach an NBQP embedding. First, if there are multiple components of $Q$, we draw edges between them to make  $Q$ connected. As no new cycles are introduced, this does not change that we have a parity embedding and the graph remains simple. 

    Hence, we might assume that $Q$ is connected, which implies that the interior of each cell is homeomorphic to an open disk or an open cross-cap.  In particular, each cell has a single boundary walk and that walk is a contractible cycle of $\Pb^2$. As we have a parity embedding, this implies that the boundary walk has even length. 

    Next, we add edges until each cell has size 4 and each becomes homeomorphic to a disk. The plan is simple; if we have a cell $f$ of size at least 6, we add diagonals to it to break it into quadrangles. Since we want simple quadrangulations, we have to be a bit more careful than this. Note, for example, that we did not assume anything about the embedding of $Q$. It is possible that a vertex or even an edge appears more than once on the boundary of $f$. We will avoid these problems by adding some extra vertices and edges. 
    
    Let $f$ be a cell of length $2k\ge 6$, or a cross-cap cell of length 4. Let $v_1,v_2,\dots, v_{2k}$ denote the boundary walk of $f$. We introduce new vertices $w_1,\dots, w_{2k}$ and place $w_i$ close to $v_i$ within $f$. Then we connect $v_i$ to $w_i$ by a short curve, and we also connect $w_i$ to $w_{i+1}$ by drawing an edge close to the boundary of $f$, see Figure \ref{fig:toquadrangulation}. These new edges cannot be loop edges, and they cannot create parallel edges (unless $Q$ is a single edge). 

    We have formed a sequence of quadrangles around $f$ and the cycle $w_1\dots w_{2k}$ bounds a new cell $f'$ of size $2k$. This cell is nice in the sense that every vertex appears at most once on the boundary of $f'$ and no two vertices of the boundary are connected outside $f'$. This also implies that $f'$ is homeomorphic to either a disk or a cross-cap. 

    If $f'$ is a disk, then we can add the edges $w_1w_4$, $w_1w_6,\dots, w_1w_{2k-2}$ to break $f'$ into quadrangles while preserving that the graph is simple, see Figure \ref{fig:toquadrangulationa}. If $f'$ is a cross-cap, we can add the edge $w_1w_3$ such that $w_1w_2w_3$ forms an essential cycle, see Figure \ref{fig:toquadrangulationb}. This will not divide the face into two faces, but it changes $f'$ so that it will be homeomorphic to a disk. Applying the whole procedure again, we can divide the face into quadrangles. Repeating this for each large face results in an embedding where each cell is of size four and each is homeomorphic to a disk. Hence, we have extended the embedding to a simple quadrangulation of the projective plane.

    If $Q$ was originally bipartite, then one of the faces was homeomorphic to a cross-cap, and we made the graph non-bipartite by creating the $w_1w_2w_3$ cycle. Hence, the resulting quadrangulation is an NBQP.

    \begin{figure}[!ht]
        \centering

    \begin{subfigure}[t]{0.48\textwidth}
    \centering
    \includegraphics[width=0.99\linewidth]{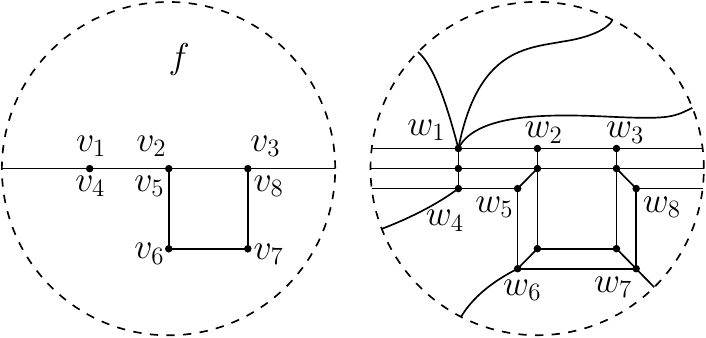}
    \caption{Extending $f'$ when it is a disk.}
    \label{fig:toquadrangulationa}
  \end{subfigure}
    \begin{subfigure}[t]{0.48\textwidth}
    \centering
    \includegraphics[width=0.99\linewidth]{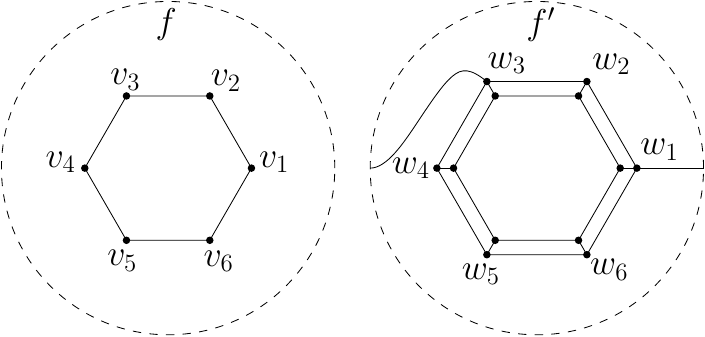}
    \caption{Turning $f'$ into a disk when it is a cross-cap.}
    \label{fig:toquadrangulationb}
  \end{subfigure}
        \caption{Quadrangulating a face of a parity embedding.}
        \label{fig:toquadrangulation}
    \end{figure}

\end{proof}

\section*{Declaration of AI usage}

\begin{center}
\begin{tabular}{l c}
\hline
\textbf{Activity} & \textbf{AI used} \\
\hline
Asking the research question     & No \\
Coming up with the approach      & No \\
Proof development                & No \\
Writing and exposition           & Yes \\
Checking for mistakes                & Yes \\
Other supporting tasks           & Yes \\
\hline
\end{tabular}
\end{center}

\paragraph{Details.}

The mathematical results are due to the author. That is, every new theorem was originally proved by the author, and a proof for each was originally written by the author. Then, AI was extensively used to clarify the ideas in the paper and to improve the presentation. This includes simplifying some of the original proofs and catching numerous inaccuracies. In some cases, the AI also suggested better definitions and clearer arguments. The AI also helped the author realize that some of the ideas have already appeared in the literature.  AI help was also used for creating some of the figures, especially the ones in the Appendix.

\section*{Acknowledgements}
I thank Andreas Holmsen, Konrad Swanepoel, László Lovász, Tibor Jordán, Márton Naszódi and Dömötör Pálvölgyi for helpful discussions and valuable comments. I am forever grateful to my wife Anna for her continuous support. 

\bibliographystyle{plain} 
\bibliography{sample}

\appendix
\section{Examples}\label{app:examples}

In this appendix, we collect several examples illustrating the NBQP-SISD correspondence. 

\begin{figure}[htbp]
    \centering

    \begin{minipage}[t]{0.48\textwidth}
        \centering
        \includegraphics[width=0.7\linewidth]{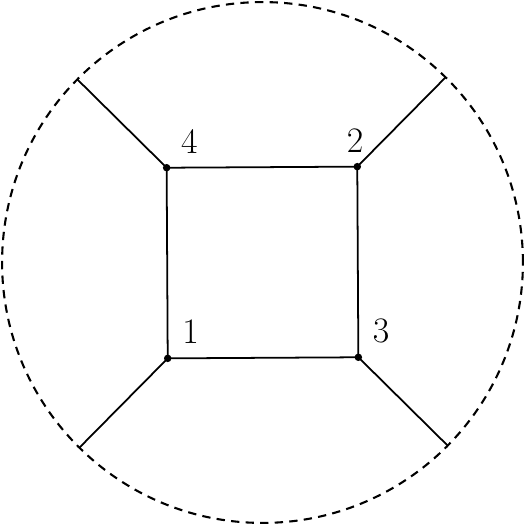}
        \par\smallskip
        (a) NBQP map
    \end{minipage}
    \hfill
    \begin{minipage}[t]{0.48\textwidth}
        \centering
        \includegraphics[width=0.7\linewidth]{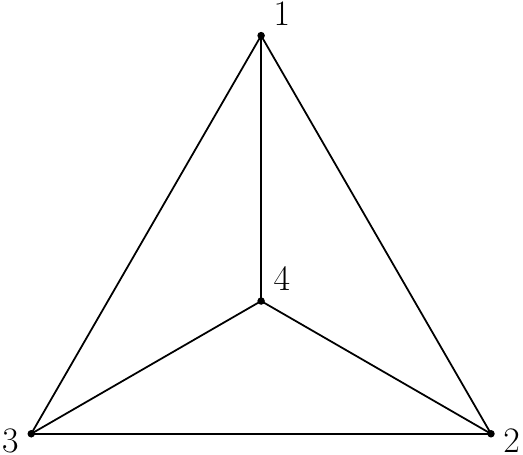}
        \par\smallskip
        (b) SISD map
    \end{minipage}

    \caption*{Example A. The unique NBQP and SISD map on 4 vertices. The two graphs are isomorphic as abstract graphs, but this is just happenstance. }
    \label{fig:example-A}
\end{figure}

\begin{figure}[htbp]
    \centering

    \begin{minipage}[t]{0.48\textwidth}
        \centering
        \includegraphics[width=0.7\linewidth]{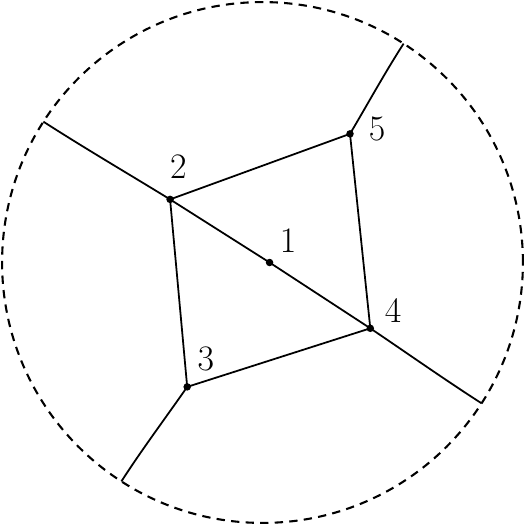}
        \par\smallskip
        (a) NBQP map
    \end{minipage}
    \hfill
    \begin{minipage}[t]{0.48\textwidth}
        \centering
        \includegraphics[width=0.7\linewidth]{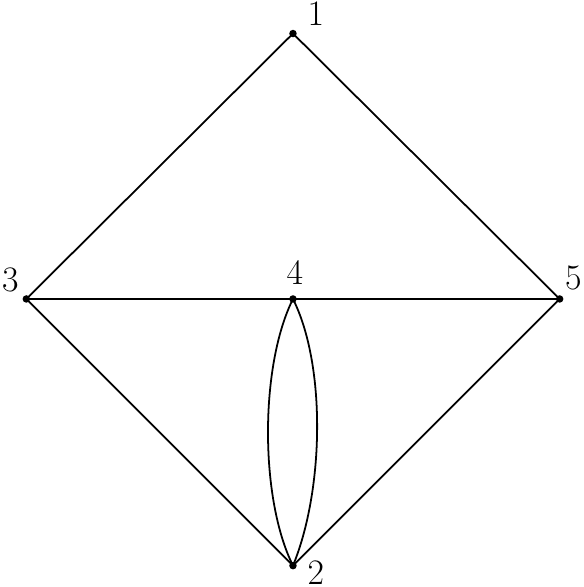}
        \par\smallskip
        (b) SISD map
    \end{minipage}

    \caption*{Example B. The unique NBQP and SISD map on 5 vertices.}
    \label{fig:example-B}
\end{figure}

\begin{figure}[htbp]
    \centering

    \begin{minipage}[t]{0.48\textwidth}
        \centering
        \includegraphics[width=0.7\linewidth]{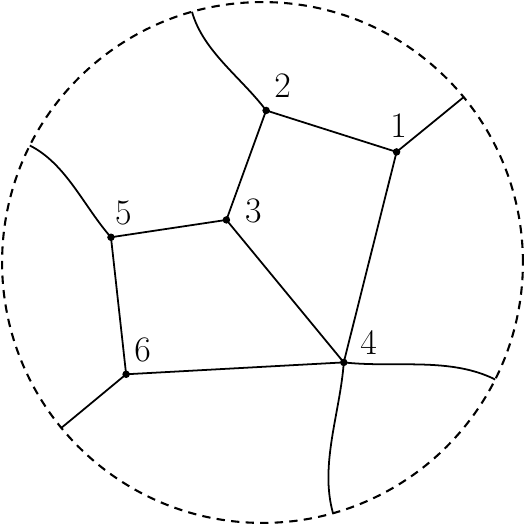}
        \par\smallskip
        (a) NBQP map
    \end{minipage}
    \hfill
    \begin{minipage}[t]{0.48\textwidth}
        \centering
        \includegraphics[width=0.7\linewidth]{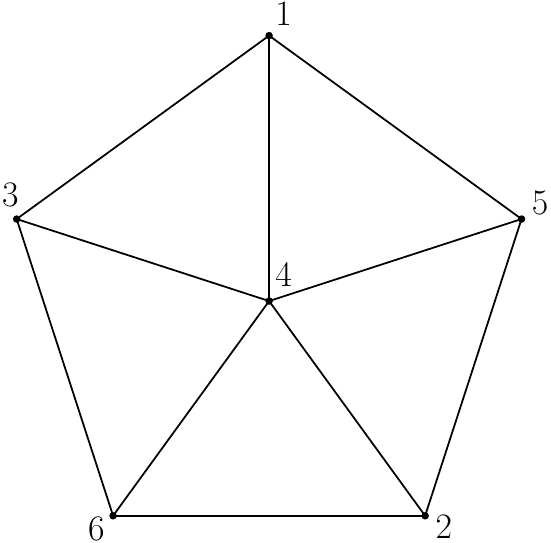}
        \par\smallskip
        (b) SISD map
    \end{minipage}

    \caption*{Example C. The unique separating 4-cycle free NBQP on 6 vertices and the corresponding 3-connected SISD map. The two graphs are isomorphic as abstract graphs, but this is just happenstance.}
    \label{fig:example-C}
\end{figure}

\begin{figure}[htbp]
    \centering

    \begin{minipage}[t]{0.48\textwidth}
        \centering
        \includegraphics[width=0.7\linewidth]{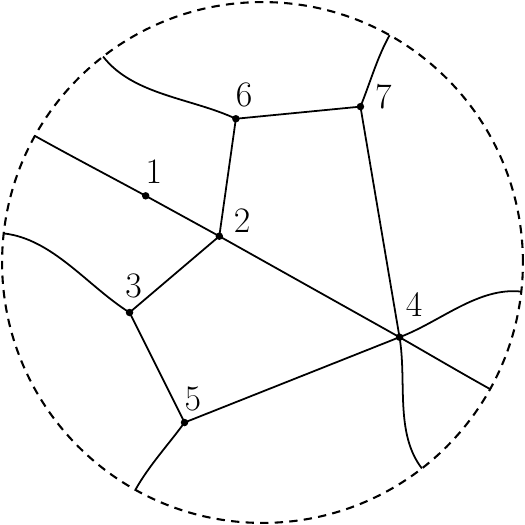}
        \par\smallskip
        (a) NBQP map
    \end{minipage}
    \hfill
    \begin{minipage}[t]{0.48\textwidth}
        \centering
        \includegraphics[width=0.7\linewidth]{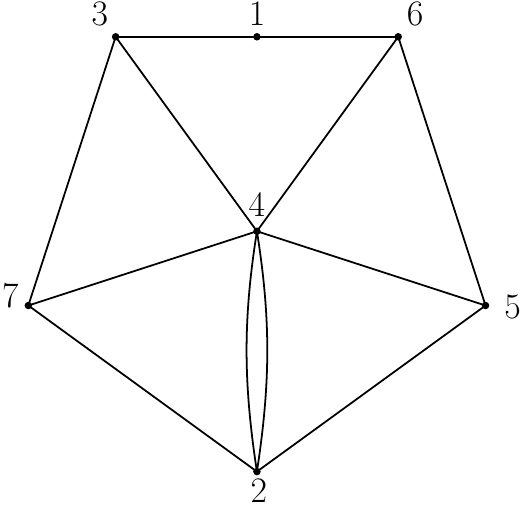}
        \par\smallskip
        (b) SISD map
    \end{minipage}

    \caption*{Example D. An NBQP on 7 vertices. The cycle 3-2-6-4 is a separating 4-cycle. The corresponding cut in the SISD map is $\{3,6\}$.}
    \label{fig:example-D}
\end{figure}

\begin{figure}[htbp]
    \centering

    \begin{minipage}[t]{0.48\textwidth}
        \centering
        \includegraphics[width=0.7\linewidth]{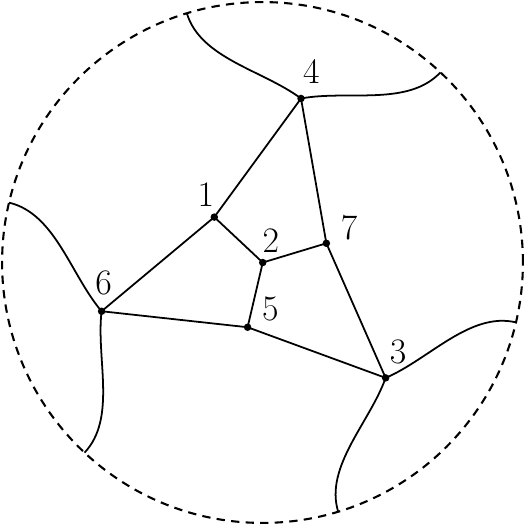}
        \par\smallskip
        (a) NBQP map
    \end{minipage}
    \hfill
    \begin{minipage}[t]{0.48\textwidth}
        \centering
        \includegraphics[width=0.7\linewidth]{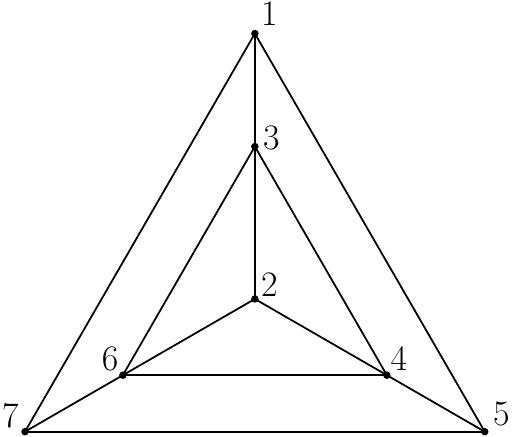}
        \par\smallskip
        (b) SISD map
    \end{minipage}

    \caption*{Example E. The unique separating 4-cycle free NBQP on 7 vertices and the corresponding 3-connected SISD map.}
    \label{fig:example-E}
\end{figure}

\begin{figure}[htbp]
    \centering

    \begin{minipage}[t]{0.48\textwidth}
        \centering
        \includegraphics[width=0.7\linewidth]{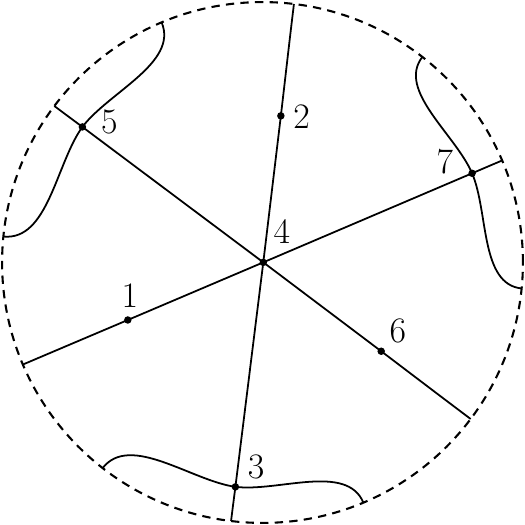}
        \par\smallskip
        (a) NBQP map
    \end{minipage}
    \hfill
    \begin{minipage}[t]{0.48\textwidth}
        \centering
        \includegraphics[width=0.7\linewidth]{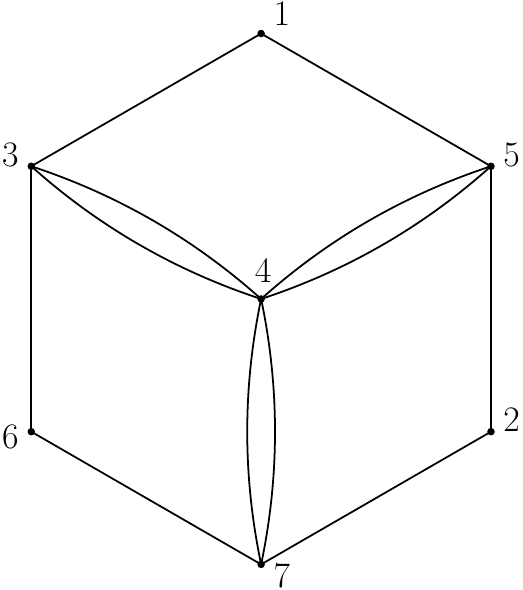}
        \par\smallskip
        (b) SISD map
    \end{minipage}

    \caption*{Example F. Another example on 7 vertices.}
    \label{fig:example-F}
\end{figure}

\begin{figure}[htbp]
    \centering

    \begin{minipage}[t]{0.48\textwidth}
        \centering
        \includegraphics[width=0.7\linewidth]{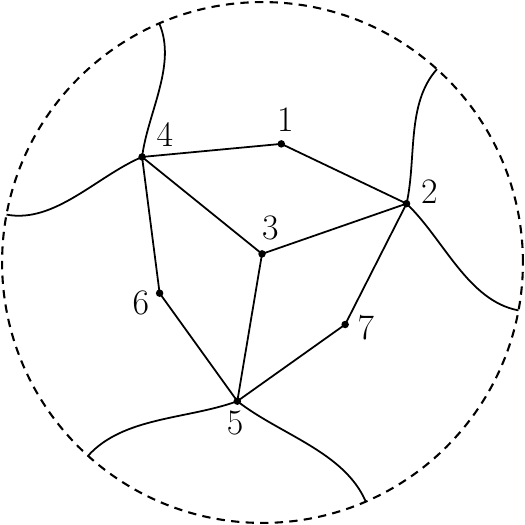}
        \par\smallskip
        (a) NBQP map
    \end{minipage}
    \hfill
    \begin{minipage}[t]{0.48\textwidth}
        \centering
        \includegraphics[width=0.7\linewidth]{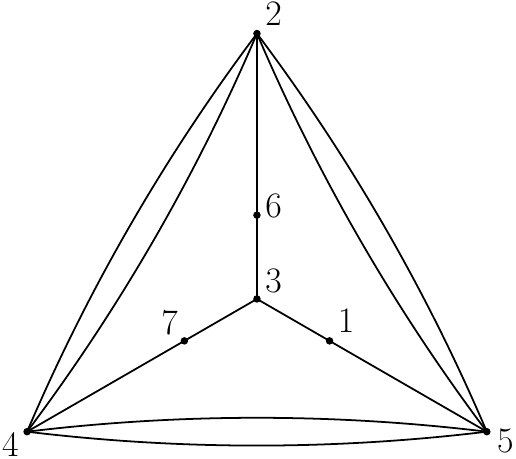}
        \par\smallskip
        (b) SISD map
    \end{minipage}

    \caption*{Example G. Another example on 7 vertices.}
    \label{fig:example-G}
\end{figure}

\begin{figure}[htbp]
    \centering

    \begin{minipage}[t]{0.48\textwidth}
        \centering
        \includegraphics[width=0.7\linewidth]{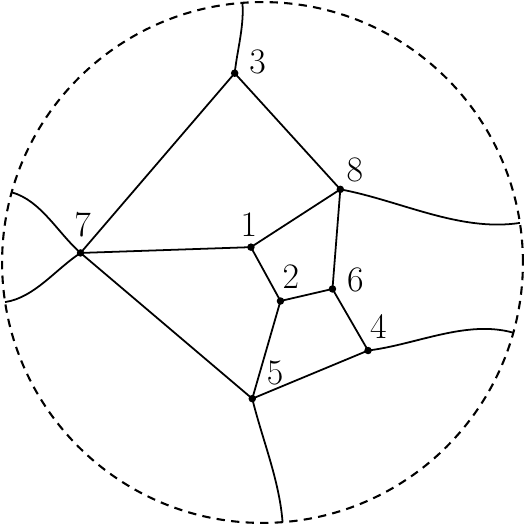}
        \par\smallskip
        (a) NBQP map
    \end{minipage}
    \hfill
    \begin{minipage}[t]{0.48\textwidth}
        \centering
        \includegraphics[width=0.7\linewidth]{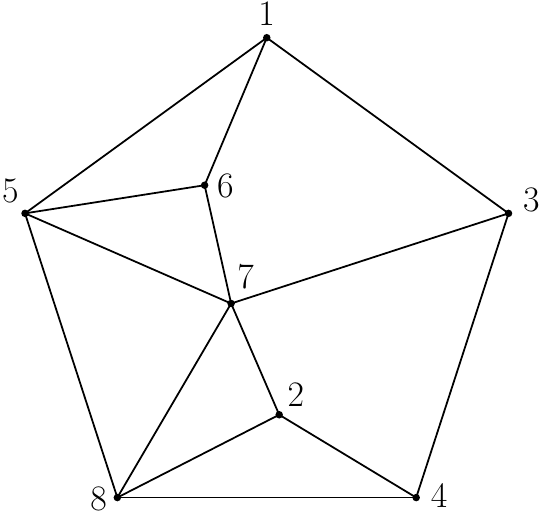}
        \par\smallskip
        (b) SISD map
    \end{minipage}

    \caption*{Example H. One of the two separating 4-cycle free NBQP maps on 8 vertices. The other one is a wheel graph.}
    \label{fig:example-H}
\end{figure}

\begin{figure}[htbp]
    \centering

    \begin{minipage}[t]{0.48\textwidth}
        \centering
        \includegraphics[width=0.7\linewidth]{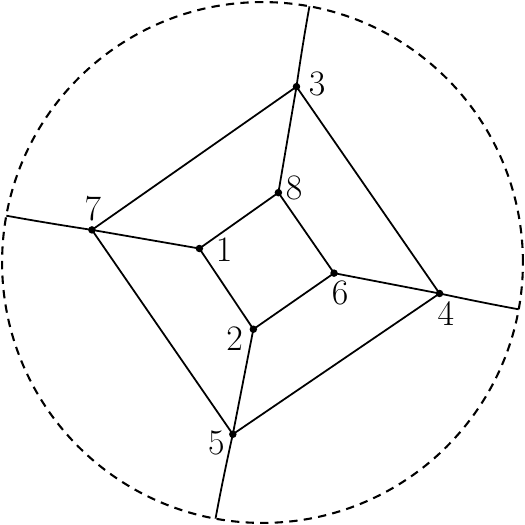}
        \par\smallskip
        (a) NBQP map
    \end{minipage}
    \hfill
    \begin{minipage}[t]{0.48\textwidth}
        \centering
        \includegraphics[width=0.7\linewidth]{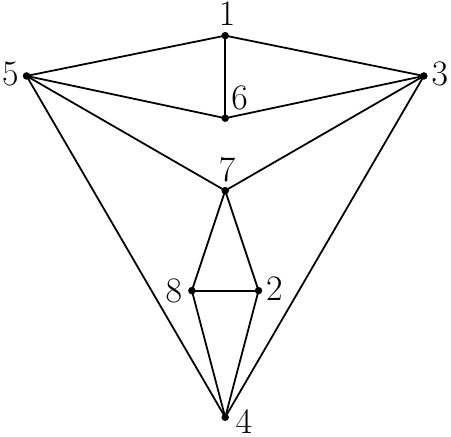}
        \par\smallskip
        (b) SISD map
    \end{minipage}

    \caption*{Example I. The smallest NBQP for which the minimum degree is at least 3, but we do have a separating 4-cycle. The corresponding SISD map does not have parallel edges, but it is not 3-connected.}
    \label{fig:example-I}
\end{figure}

\begin{figure}[htbp]
    \centering

    \begin{minipage}[t]{0.48\textwidth}
        \centering
        \includegraphics[width=0.7\linewidth]{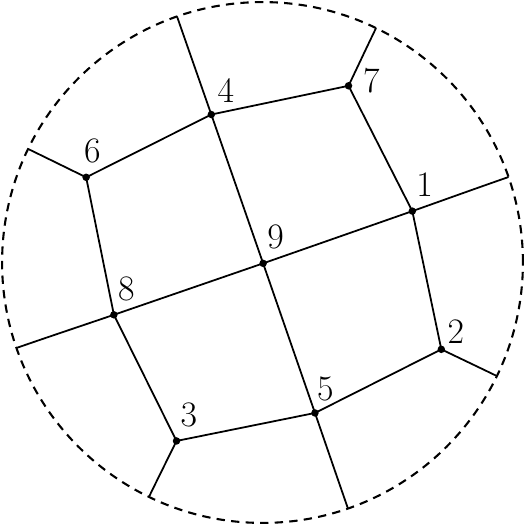}
        \par\smallskip
        (a) NBQP map
    \end{minipage}
    \hfill
    \begin{minipage}[t]{0.48\textwidth}
        \centering
        \includegraphics[width=0.7\linewidth]{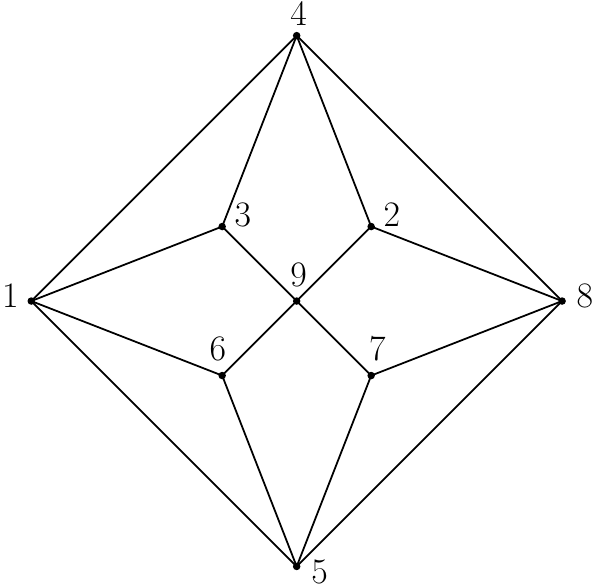}
        \par\smallskip
        (b) SISD map
    \end{minipage}

    \caption*{Example J. A highly symmetric example on 9 vertices.}
    \label{fig:example-J}
\end{figure}

\begin{figure}[htbp]
    \centering

    \begin{minipage}[t]{0.48\textwidth}
        \centering
        \includegraphics[width=0.7\linewidth]{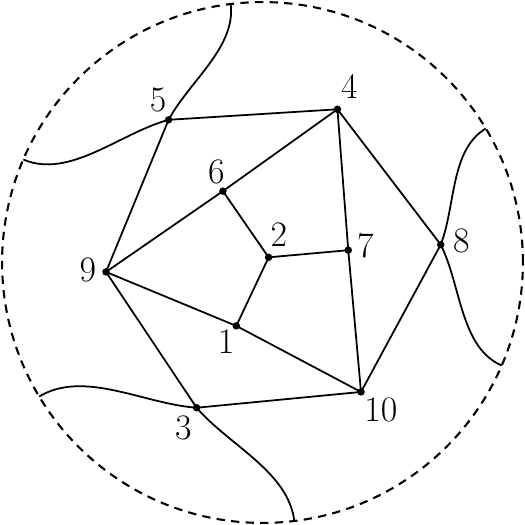}
        \par\smallskip
        (a) NBQP map
    \end{minipage}
    \hfill
    \begin{minipage}[t]{0.48\textwidth}
        \centering
        \includegraphics[width=0.7\linewidth]{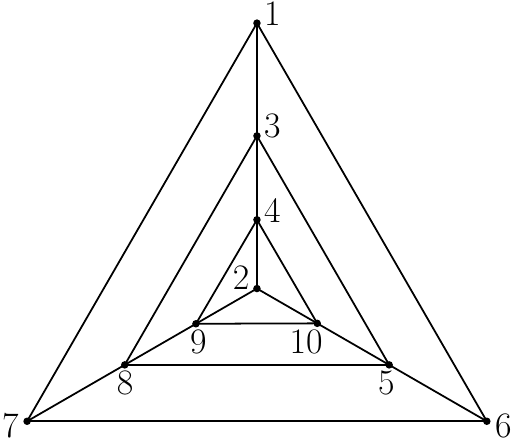}
        \par\smallskip
        (b) SISD map
    \end{minipage}

    \caption*{Example K. A highly symmetric example on 10 vertices.}
    \label{fig:example-K}
\end{figure}

\begin{figure}[htbp]
    \centering

    \begin{minipage}[t]{0.48\textwidth}
        \centering
        \includegraphics[width=0.7\linewidth]{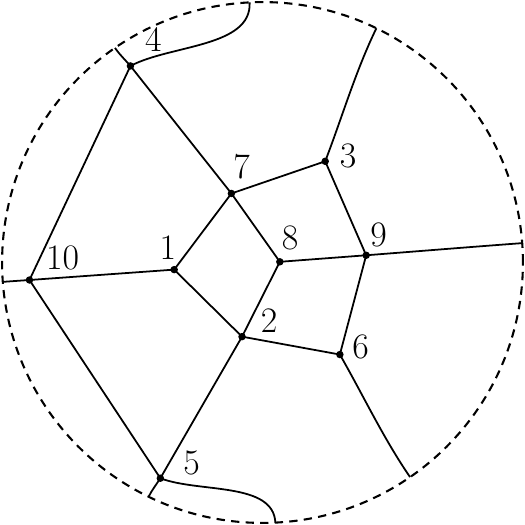}
        \par\smallskip
        (a) NBQP map
    \end{minipage}
    \hfill
    \begin{minipage}[t]{0.48\textwidth}
        \centering
        \includegraphics[width=0.7\linewidth]{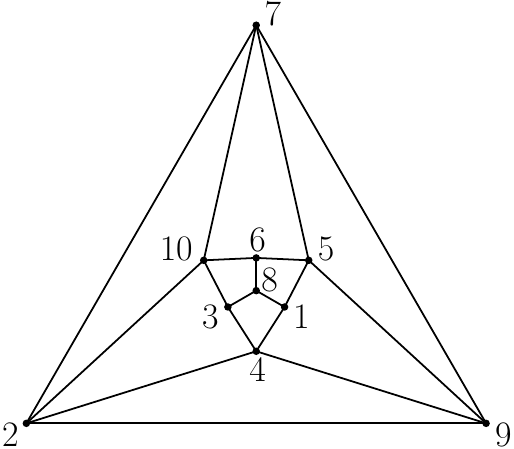}
        \par\smallskip
        (b) SISD map
    \end{minipage}

    \caption*{Example L. Another highly symmetric example on 10 vertices.}
    \label{fig:example-L}
\end{figure}

\end{document}